\documentclass{amsart}
\usepackage{
    amsmath,
    amssymb,
    relsize,
    color,
    cprotect,
    floatrow,
    graphicx,
    hyperref,
    url,
    wrapfig,
    subcaption,
    tikz,
    ifthen,
    calc,
    MnSymbol,
    multicol
}
\usetikzlibrary{math,calc}

\def\N{\mathbb{N}}
\def\Q{\mathbb{Q}}

\def\R{\mathbb{R}}

\newtheorem{theorem}{\bf Theorem}[section]
\newtheorem*{theorem*}{\bf Theorem}

\newtheorem{lemma}[theorem]{\bf Lemma}

\newtheorem{question}[theorem]{\bf Question}
\newtheorem*{question*}{\bf Question}
\newtheorem{proposition}[theorem]{\bf Proposition}
\newtheorem*{proposition*}{\bf Proposition}
\newtheorem{corollary}[theorem]{\bf Corollary}
\newtheorem*{corollary*}{\bf Corollary}

\newcommand{\potholder}[2]{
\begin{tikzpicture}[line width=1.5pt,black,scale=#2]
\tikzmath{\n = #1;}
\foreach \x [evaluate=\x as \y using {int(mod(\x,2))}] 
in {1,...,\n}{
\draw (\x+\y-0.5,\n+1) to[out={180*\x},in=90] (\x,\n+0.5) -- (\x,0.5) to[in={180*\x+180},out=270] (\x-\y+0.5,0);
\draw (\n+1,\x+\y-0.5) to[out={180*\x+90},in=0] (\n+0.5,\x) -- (0.5,\x) to[in={180*\x-90},out=180] (0,\x-\y+0.5);
}
\draw (0,0.5) to[in=180,out=270] (0.5,0);
\draw (\n+1,\n+0.5) to[in=0,out=90] (\n+0.5,\n+1);
\pgfresetboundingbox \clip (-1,-1) rectangle (\n+2,\n+2);
\end{tikzpicture}}

\newcommand{\starlet}[2]{
\begin{tikzpicture}[line width=1.5pt,black,scale=#2]
\tikzmath{\n = #1; \m=\n*2+1;}
\foreach \x in {1,...,\m}{
\draw ({(90+360.0*\x*\n/\m}:1) -- ({(90+360.0*(\x+1)*\n/\m}:1);
}
\pgfresetboundingbox \clip (-1.1,-1.1) rectangle (1.1,1.1);
\end{tikzpicture}}

\pgfmathdeclarefunction{ball}{1}{%
  \pgfmathparse{min(2*(#1-int(#1)),2-2*(#1-int(#1)))}%
}

\newcommand{\billiard}[2]{
\begin{tikzpicture}[line width=1.5pt,black,scale=#2]
\tikzmath{\n = #1; \m=\n*2+1; \nn = 1.0/\n; \mm = 1.0/\m; \nm = \m*\n;}
\foreach \x in {0,0.25,...,\nm}{
\draw ({ball((\x+0.25) * \mm)},{ball(\x * \nn) * 0.8}) 
-- ({ball((\x+0.5) * \mm)},{ball((\x+0.25) * \nn) * 0.8});
}
\pgfresetboundingbox \clip (-0.1,-0.1) rectangle (1.1,0.9);
\end{tikzpicture}}

\newcommand{\starbilliard}[1]{
\begin{tikzpicture}[line width=1.5pt,black,scale=#1]
\tikzmath{\n = 3; \m=\n*2+1;}
\foreach \x in {1,...,\m}
\draw[black] ({(90+360.0*\x*\n/\m}:1) -- ({(90+360.0*(\x+1)*\n/\m}:1);
\fill[white] (-1.1,-1.1) rectangle (1.1,0.0);
\foreach \x in {1,...,\m}
\draw[darkgray] ({(-90+360.0*\x*\n/\m}:1) -- ({(-90+360.0*(\x+1)*\n/\m}:1);
\draw[lightgray,dashed] (-1.1,0) -- (1.1,0);
\foreach \x in {-6,-4,-2,2,4,6}
\node[black,align=center] at ({tan(180.0/14)/sin(180.0*\x/14)*0.975},0) {\Large\textbullet};
\fill[white,fill opacity=0.5] (-1.1,-1.1) rectangle (1.1,-0.005);
\pgfresetboundingbox \clip (-1.0,-1.05) rectangle (1.0,1.05);
\end{tikzpicture}}

\newcommand{\starincrease}[1]{
\begin{tikzpicture}[scale=#1]
\draw[lightgray, line width=3pt] (-0.8,0.4) -- (0.5,-1) -- (0,0.8) -- (-0.5,-1) -- (0.8,0.4) -- (-1,-0.25) -- (1,-0.25) -- (-0.8,0.4) -- (0.5,-1);
\draw[black, line width=1.5pt] (-0.15,-0.3) -- (-0.8,0.4) -- (1,-0.25) -- (-1,-0.25) -- (0.8,0.4) -- (0.15,-0.3) to[in=315,out=225] (-0.15,-0.3);
\pgfresetboundingbox \clip (-1.0,-1.0) rectangle (1.0,1.0);
\end{tikzpicture}}

\newcommand{\figureeight}[1]{
\begin{tikzpicture}[line width = 4pt,scale=#1]
\draw[shorten <=8pt,gray,double=lightgray,double distance=1pt,line width = 2pt] 
(-0.5,0) .. controls +(-1,-0.5) and +(-1,0.5) .. (0,-1.75);
\draw[gray,double=lightgray,double distance=1pt,line width = 2pt]
(0,-1.75) .. controls +(2,-1) and +(0,-0.5) .. (2.25,-0.75);
\draw[shorten >=8pt] 
(2.25,-0.75) .. controls +(0,+0.5) and +(1,-0.5) .. (1.25,0.75);
\draw[shorten <=8pt] 
(1.25,0.75) .. controls +(-0.5,0.25) and +(0.75,-0.25) .. (-1,1.75);
\draw[shorten >=8pt] 
(-1,1.75) .. controls +(-2.25,0.75) and +(-2,-1) .. (0,-1.75);
\draw[shorten <=8pt] 
(0,-1.75) .. controls +(1,0.5) and +(0.5,-0.5) .. (-0.5,0);
\draw
(-0.5,0) .. controls +(-0.2,0.2) and +(0,-0.3) .. (-1.25,1);
\draw[shorten >=8pt,gray,double=lightgray,double distance=1pt,line width = 2pt] 
(-1.25,1) .. controls +(0,+0.2) and +(-0.2,-0.2) .. (-1,1.75);
\draw[shorten <=8pt,gray,double=lightgray,double distance=1pt,line width = 2pt] 
(-1,1.75)  .. controls +(0.75,0.75) and +(1,1.5) .. (1.25,0.75);
\draw[shorten >=8pt,gray,double=lightgray,double distance=1pt,line width = 2pt] 
(1.25,0.75) .. controls +(-0.5,-0.75) and +(0.5,0.25) .. (-0.5,0);
\node[black,align=center] at (-1.25,1) {\Huge\textbullet};
\node[black,align=center] at (2.25,-0.75) {\Huge\textbullet};
\pgfresetboundingbox \clip (-3,-2.5) rectangle (3,2.5);
\end{tikzpicture}}

\newcommand{\xmove}[1]{
\begin{tikzpicture}[line width = 3pt,scale=#1]
\draw (-3.75,0) -- (-1.25,0);
\draw (-3,-1) -- (-3,1);
\draw (-2,-1) -- (-2,1);
\node (0,0) {\Huge $\to$};
\draw (3.75,0) -- (1.25,0);
\draw (3,-1) .. controls +(90:.3) and +(-90:.3) .. (2.2,0) .. controls +(90:.3) and +(-90:.3) .. (3,1);
\draw (2,-1) .. controls +(90:.3) and +(-90:.3) .. (2.8,0) .. controls +(90:.3) and +(-90:.3) .. (2,1);
\pgfresetboundingbox \clip (-5,-1.5) rectangle (5,1.5);
\end{tikzpicture}}

\newcommand{\xmovelifted}[1]{
\begin{tikzpicture}[line width = 3pt,white,double=black,double distance = 3pt,scale=#1]
\draw[double=black] (-3,-1) -- (-3,1);
\draw[double=black] (-3.75,0) -- (-1.25,0);
\draw[double=black] (-2,-1) -- (-2,1);
\node[black] (0,0) {\Huge $\to$};
\draw[double=black] (2,-1) .. controls +(90:.3) and +(-90:.3) .. (2.8,0) .. controls +(90:.3) and +(-90:.3) .. (2,1);
\draw[double=black] (3.75,0) -- (1.25,0);
\draw[double=black] (3,-1) .. controls +(90:.3) and +(-90:.3) .. (2.2,0) .. controls +(90:.3) and +(-90:.3) .. (3,1);
\pgfresetboundingbox \clip (-5,-1.25) rectangle (5,1.25);
\end{tikzpicture}}

\newcommand{\potholderfigure}[1]{
\begin{tikzpicture}[thick,scale=#1]
\foreach \x in {2,4} \draw [white,line width=3.0pt,double=gray,double distance=5.0pt] (\x,3.5) to[in=0,out=90] (\x-0.5,4);
\foreach \x in {1,3} \draw [white,line width=3.0pt,double=gray,double distance=5.0pt] (\x,3.5) to[in=180,out=90] (\x+0.5,4);
\foreach \x in {0,2} \draw [white,line width=3.0pt,double=gray,double distance=5.0pt] (\x,0.5) to[in=180,out=270] (\x+0.5,0);
\foreach \x in {1,3} \draw [white,line width=3.0pt,double=gray,double distance=5.0pt] (\x,0.5) to[in=0,out=270] (\x-0.5,0);
\foreach \x/\y in {1/1,1/2,2/2,2/3,3/1,3/2}{
\draw [white,line width=3.0pt,double=gray,double distance=5.0pt] (\x,\y-0.51) to (\x,\y+0.51);
\draw [white,line width=3.0pt,double=black,double distance=4.0pt] (\x-0.51,\y) to (\x+0.51,\y);}
\foreach \x/\y in {1/3,2/1,3/3}{
\draw [white,line width=3.0pt,double=black,double distance=4.0pt] (\x-0.51,\y) to (\x+0.51,\y);
\draw [white,line width=3.0pt,double=gray,double distance=5.0pt] (\x,\y-0.51) to (\x,\y+0.51);}
\foreach \y in {1,3} \draw [white,line width=3.0pt,double=black,double distance=4.0pt] (0.5,\y) to[in=90,out=180] (0,\y-0.5);
\foreach \y in {2} \draw [white,line width=3.0pt,double=black,double distance=4.0pt] (0.5,\y) to[in=270,out=180] (0,\y+0.5); 
\foreach \y in {1,3} \draw [white,line width=3.0pt,double=black,double distance=4.0pt] (3.5,\y) to[in=270,out=0] (4,\y+0.5);
\foreach \y in {2} \draw [white,line width=3.0pt,double=black,double distance=4.0pt] (3.5,\y) to[in=90,out=0] (4,\y-0.5);
\node[black] at (0,0.5) {\Huge \textbullet};
\node[black] at (4,3.5) {\Huge \textbullet};
\pgfresetboundingbox \clip (-1,-0.5) rectangle (5,4.5);
\end{tikzpicture}}

\newcommand{\borromean}[1]{
\begin{tikzpicture}[scale=#1]
\draw[white,line width=3.0pt,double=gray,double distance=5.0pt] (-90:1.2) .. controls +(-0:0.6) and +(-60:0.6) .. (-0:0.7);
\draw[gray,line width=2.0pt,double=lightgray,double distance=1.0pt] (-90:1.2) .. controls +(-0:0.6) and +(-60:0.6) .. (-0:0.7);
\draw[white,line width=3.0pt,double=black,double distance=4.0pt] (-210:1.2) .. controls +(-120:0.6) and +(-180:0.6) .. (-120:0.7);
\draw[white,line width=3.0pt,double=black,double distance=4.0pt] (-330:1.2) .. controls +(-240:0.6) and +(60:0.6) .. (-240:0.7);
\draw[white,line width=3.0pt,double=gray,double distance=5.0pt] (-90:1.2) .. controls +(-180:0.6) and +(-120:0.6) .. (-180:0.7);
\draw[gray,line width=2.0pt,double=lightgray,double distance=1.0pt] (-90:1.2) .. controls +(-180:0.6) and +(-120:0.6) .. (-180:0.7);
\draw[white,line width=3.0pt,double=gray,double distance=5.0pt] (-210:1.2) .. controls +(-300:0.6) and +(120:0.6) .. (-300:0.7);
\draw[gray,line width=2.0pt,double=lightgray,double distance=1.0pt] (-210:1.2) .. controls +(-300:0.6) and +(120:0.6) .. (-300:0.7);
\draw[white,line width=3.0pt,double=black,double distance=4.0pt] (-330:1.2) .. controls +(-60:0.6) and +(0:0.6) .. (-60:0.7);
\draw[white,line width=3.0pt,double=black,double distance=4.0pt] (-90:-0.4) .. controls +(-180:0.3) and +(60:0.3) .. (-180:0.7);
\draw[white,line width=3.0pt,double=gray,double distance=5.0pt] (-210:-0.4) .. controls +(-300:0.3) and +(-60:0.3) .. (-300:0.7);
\draw[gray,line width=2.0pt,double=lightgray,double distance=1.0pt] (-210:-0.4) .. controls +(-300:0.3) and +(-60:0.3) .. (-300:0.7);
\draw[white,line width=3.0pt,double=black,double distance=4.0pt] (-330:-0.4) .. controls +(-60:0.3) and +(180:0.3) .. (-60:0.7);
\draw[white,line width=3.0pt,double=black,double distance=4.0pt] (-90:-0.4) .. controls +(-0:0.3) and +(120:0.3) .. (-0:0.7);
\draw[white,line width=3.0pt,double=gray,double distance=5.0pt] (-210:-0.4) .. controls +(-120:0.3) and +(0:0.3) .. (-120:0.7);
\draw[gray,line width=2.0pt,double=lightgray,double distance=1.0pt] (-210:-0.4) .. controls +(-120:0.3) and +(0:0.3) .. (-120:0.7);
\draw[white,line width=3.0pt,double=gray,double distance=5.0pt] (-330:-0.4) .. controls +(-240:0.3) and +(-120:0.3) .. (-240:0.7);
\draw[gray,line width=2.0pt,double=lightgray,double distance=1.0pt] (-330:-0.4) .. controls +(-240:0.3) and +(-120:0.3) .. (-240:0.7);
\node[black] at (-180:0.7) {\Huge \textbullet};
\node[black] at (-0:0.7) {\Huge \textbullet};
\node[black] at (-120:0.7) {\Huge \textbullet};
\node[black] at (-210:1.2) {\Huge \textbullet};
\node[black] at (-240:0.7) {\Huge \textbullet};
\node[black] at (-330:-0.4) {\Huge \textbullet};
\pgfresetboundingbox \clip (-1.25,-1.25) rectangle (1.25,1.25);
\end{tikzpicture}}

\newcommand{\potholderlinkfigure}[1]{
\begin{tikzpicture}[thick,scale=#1]
\foreach \x [evaluate=\x as \y using {1-2*int(mod(\x,2))}] in {1,...,4}
\foreach \z in {-0.25,0,0.25} {
\draw [white,line width=1.0pt,double=gray,double distance=3.0pt] 
(\x+\z*\y,3.5) to[in=90-90*\y,out=90] (\x-0.5*\y,4+\z);
\draw [white,line width=1.0pt,double=gray,double distance=3.0pt] 
(\x+\z*\y-1,0.5) to[in=90-90*\y,out=-90] (\x-0.5*\y-1,-\z);}
\foreach \x [evaluate=\x as \y using {1-2*int(mod(\x,2))}] in {1,...,3}
\foreach \z in {-0.25,0,0.25} {
\draw [white,line width=1.0pt,double=black,double distance=2.0pt] 
(3.5,\x+\z*\y) to[in=90*\y,out=0] (4+\z,\x-0.5*\y);
\draw [white,line width=1.0pt,double=black,double distance=2.0pt] 
(0.5,4-\x-\z*\y) to[in=-90*\y,out=180] (-\z,4-\x+0.5*\y);}
\foreach \x in {1,2,3} \foreach \y in {1,2,3}
\foreach \xa/\xb/\xc in {-0.5/-0.25/-0.15,-0.15/0/0.15,0.15/0.25/0.5}
\foreach \ya/\yb/\yc in {-0.5/-0.25/-0.15,-0.15/0/0.15,0.15/0.25/0.5}{
\draw [white,line width=1.0pt,double=black,double distance=2.0pt] 
(\x+\xa-0.01,\y+\yb) -- (\x+\xc+0.01,\y+\yb);
\draw [white,line width=1.0pt,double=gray,double distance=3.0pt] 
(\x+\xb,\y+\ya-0.01) -- (\x+\xb,\y+\yc+0.01);
\draw [white,line width=1.0pt,double=black,double distance=2.0pt,opacity=int(rand+1)] 
(\x+\xa-0.01,\y+\yb) -- (\x+\xc+0.01,\y+\yb);
}
\node[black] at (0,0.5) {\LARGE \textbullet};
\node[black] at (4,3.5) {\LARGE \textbullet};
\node[black] at (0.25,0.5) {\LARGE \textbullet};
\node[black] at (4.25,3.5) {\LARGE \textbullet};
\node[black] at (-0.25,0.5) {\LARGE \textbullet};
\node[black] at (3.75,3.5) {\LARGE \textbullet};
\pgfresetboundingbox \clip (-0.3,-0.3) rectangle (4.3,4.3);
\end{tikzpicture}}

\newcommand{\xmoves}[1]{
\begin{tikzpicture}[line width = 4pt,scale=#1]
\draw[gray,line width=2.0pt,double=lightgray,double distance=1.0pt] (-5.5,0) -- (-5.0,0);
\draw (-5.0,0) -- (-1,0);
\draw[gray,line width=2.0pt,double=lightgray,double distance=1.0pt] (-4.3,-1) -- (-4.3,1);
\draw[gray,line width=2.0pt,double=lightgray,double distance=1.0pt] (-3.4,-1) -- (-3.4,1);
\draw[gray,line width=2.0pt,double=lightgray,double distance=1.0pt] (-2.5,-1) -- (-2.5,1);
\draw (-1.6,-1) -- (-1.6,1);
\node (0,0) {\Huge $\to$};
\draw (5.5,0) -- (2.4,0);
\draw[gray,line width=2.0pt,double=lightgray,double distance=1.0pt] (2.4,0) -- (1,0);
\draw (5,-1) .. controls +(90:.4) and +(-90:.4) .. (1.6,0) .. controls +(90:.4) and +(-90:.4) .. (5,1);
\draw[gray,line width=2.0pt,double=lightgray,double distance=1.0pt] (4,-1) .. controls +(90:.3) and +(-90:.3) .. (5,0) .. controls +(90:.3) and +(-90:.3) .. (4,1);
\draw[gray,line width=2.0pt,double=lightgray,double distance=1.0pt] (3,-1) .. controls +(90:.3) and +(-90:.3) .. (4,0) .. controls +(90:.3) and +(-90:.3) .. (3,1);
\draw[gray,line width=2.0pt,double=lightgray,double distance=1.0pt] (2,-1) .. controls +(90:.3) and +(-90:.3) .. (3,0) .. controls +(90:.3) and +(-90:.3) .. (2,1);
\node[black] at (2.4,-0.03) {\Huge \textbullet};
\node[black] at (-5.0,-0.03) {\Huge \textbullet};
\pgfresetboundingbox \clip (-5,-1.5) rectangle (5,1.5);
\end{tikzpicture}}

\newcommand{\standard}[1]{
\begin{tikzpicture}[line width = 4pt,scale=#1]
\draw[shorten <=8pt,gray,double=lightgray,double distance=1pt,line width = 2pt] 
(-0.5,0) .. controls +(-1,-0.5) and +(-1,0.5) .. (0,-1.75);
\draw[gray,double=lightgray,double distance=1pt,line width = 2pt]
(0,-1.75) .. controls +(2,-1) and +(0,-0.5) .. (2.25,-0.75);
\draw[shorten >=8pt] 
(2.25,-0.75) .. controls +(0,+0.5) and +(1,-0.5) .. (1.25,0.75);
\draw[shorten <=8pt] 
(1.25,0.75) .. controls +(-0.5,0.25) and +(0.75,-0.25) .. (-1,1.75);
\draw[shorten >=8pt] 
(-1,1.75) .. controls +(-2.25,0.75) and +(-2,-1) .. (0,-1.75);
\draw[shorten <=8pt] 
(0,-1.75) .. controls +(1,0.5) and +(0.5,-0.5) .. (-0.5,0);
\draw
(-0.5,0) .. controls +(-0.2,0.2) and +(0,-0.3) .. (-1.25,1);
\draw[shorten >=8pt,gray,double=lightgray,double distance=1pt,line width = 2pt] 
(-1.25,1) .. controls +(0,+0.2) and +(-0.2,-0.2) .. (-1,1.75);
\draw[shorten <=8pt,gray,double=lightgray,double distance=1pt,line width = 2pt] 
(-1,1.75)  .. controls +(0.75,0.75) and +(1,1.5) .. (1.25,0.75);
\draw[shorten >=8pt,gray,double=lightgray,double distance=1pt,line width = 2pt] 
(1.25,0.75) .. controls +(-0.5,-0.75) and +(0.5,0.25) .. (-0.5,0);
\draw[dashed,black,line width = 4pt] 
(2.15,0) .. controls +(-0.5,0) and +(0.5,-0.5) .. (-0.25,0.65);
\draw[dashed,gray,line width = 5pt] 
(1.0,-2.1) .. controls +(-3.5,-1) and +(-3.5,-1.5) .. (-0.25,0.65);
\draw[darkgray,line width = 2pt] 
(-1.25,1) -- (-0.25,0.65);
\node[darkgray,align=center] at (-1.25,1) {\Huge\textbullet};
\node[darkgray] at (-1.5,1) {\Large \textbf{t}};
\node[darkgray,align=center] at (2.25,-0.75) {\Huge\textbullet};
\node[darkgray] at (2.5,-0.75) {\Large \textbf{s}};
\node[darkgray,align=center] at (2.15,0) {\Huge\textbullet};
\node[darkgray] at (2.4,0) {\Large \textbf{a}};
\node[darkgray,align=center] at (1.0,-2.1) {\Huge\textbullet};
\node[darkgray] at (1.0,-1.75) {\Large \textbf{b}};
\node[darkgray,align=center] at (-0.25,0.65) {\Huge\textbullet};
\node[darkgray] at (0,0.9) {\Large \textbf{t'}};
\node[black] at (-2.5,0.5) {\Large \textbf{A}};
\node[darkgray] at (0,2.5) {\Large \textbf{B}};
\node[black] at (1.25,-0.3) {\Large \textbf{A'}};
\node[darkgray] at (-2.75,-1.5) {\Large \textbf{B'}};
\pgfresetboundingbox \clip (-3,-2.75) rectangle (3,2.75);
\end{tikzpicture}}

\newcommand{\meanderEDABC}{
\begin{tikzpicture}[thick,scale=1.00]
\foreach \y in {1,3,4}{
\draw [gray,line width=2.0pt,double=lightgray,double distance=2.0pt] (-0.5,\y) to (+0.5,\y);
\draw [white,line width=4.0pt,double=black,double distance=4.0pt] (0,\y-0.5) to (0,\y+0.5);
}
\foreach \y in {2,5}{
\draw [white,line width=4.0pt,double=black,double distance=4.0pt] (0,\y-0.5) to (0,\y+0.5);
\draw [white,line width=4.0pt,double=lightgray,double distance=4.0pt] (-0.5,\y) to (0.5,\y);
\draw [gray,line width=2.0pt,double=lightgray,double distance=2.0pt] (-0.5,\y) to (0.5,\y);
}
\draw [gray,line width=2.0pt,double=lightgray,double distance=2.0pt] 
(-0.5,3) to[in=180,out=180] (-0.5,6) to (0,6);
\draw [gray,line width=2.0pt,double=lightgray,double distance=2.0pt] 
(-0.5,4) to[in=180,out=180] (-0.5,5);
\draw [gray,line width=2.0pt,double=lightgray,double distance=2.0pt] 
(-0.5,1) to[in=180,out=180] (-0.5,2);
\draw [gray,line width=2.0pt,double=lightgray,double distance=2.0pt] 
(0.5,5) to[in=0,out=0] (0.5,0) to (0,0);
\draw [gray,line width=2.0pt,double=lightgray,double distance=2.0pt] 
(0.5,4) to[in=0,out=0] (0.5,1);
\draw [gray,line width=2.0pt,double=lightgray,double distance=2.0pt] 
(0.5,3) to[in=0,out=0] (0.5,2);
\draw [white,line width=4.0pt,double=black,double distance=4.0pt] (0,0.1) to (0,0.5);
\draw [white,line width=4.0pt,double=black,double distance=4.0pt] (0,5.9) to (0,5.5);
\node[black] at (0,0) {\Huge \textbullet};
\node[black] at (-0.3,0) {\Large $s$};
\node[black] at (0,6) {\Huge \textbullet};
\node[black] at (0.3,6) {\Large $t$};
\node[black] at (-0.3,0.5) {\Large \textbf{$A$}};
\node[darkgray] at (0.5,-0.3) {\Large \textbf{$B$}};
\foreach \y in {1,2,3,4,5}{ 
\draw[-|] (-2.0,\y-0.5) -- (-2.0,\y);
\draw[|-] (-2.0,\y) -- (-2.0,\y+0.5);
\node[black] at (-2.5,\y) {\Large \y};
}
\pgfresetboundingbox \clip (-2.5,-1.25) rectangle (2.5,6);
\end{tikzpicture}}

\newcommand{\potholderEDABC}{
\begin{tikzpicture}[thick,scale=0.6]
\foreach \x in {2,4,6}
\draw [white,line width=3.0pt,double=gray,double distance=3.0pt] 
(\x,9.5) to (\x,10) to[in=0,out=90] (\x-0.5,10.5);
\foreach \x in {1,3,5}
\draw [white,line width=3.0pt,double=gray,double distance=3.0pt] 
(\x,9.5) to (\x,10) to[in=180,out=90] (\x+0.5,10.5);
\foreach \x in {0,2,4}
\draw [white,line width=3.0pt,double=gray,double distance=3.0pt] 
(\x,0.5) to (\x,0) to[in=180,out=270] (\x+0.5,-0.5);
\foreach \x in {1,3,5}
\draw [white,line width=3.0pt,double=gray,double distance=3.0pt] 
(\x,0.5) to (\x,0) to[in=0,out=270] (\x-0.5,-0.5);
\foreach \x/\y in {
2/8,2/9,
3/2,3/3,3/4,3/5,3/6,3/7,3/8,3/9,
4/4,4/5,4/6,4/7,4/8,4/9,
5/6,5/7,5/8,5/9
}{
\draw [white,line width=3.0pt,double=gray,double distance=3.0pt] (\x,\y-0.5) to (\x,\y+0.5);
\draw [white,line width=3.0pt,double=black,double distance=3.0pt] (\x-0.5,\y) to (\x+0.5,\y);
}
\foreach \x/\y in {1/9,4/3}{
\draw [white,line width=3.0pt,double=black,double distance=3.0pt] (\x-0.5,\y) to (\x+0.5,\y);
\draw [white,line width=3.0pt,double=gray,double distance=3.0pt] (\x,\y-0.5) to (\x,\y+0.5);
}
\foreach \x/\y in {2/7,3/1,5/5}{
\draw [white,line width=3.0pt,double=gray,double distance=3.0pt] (\x,\y-0.5) to (\x,\y+0.5);
\draw [white,line width=3.0pt,double=black,double distance=3.0pt] (\x-0.5,\y) to (\x+0.5,\y);
}
\foreach \x/\y in {
1/1,1/2,1/3,1/4,1/5,1/6,1/7,1/8,
2/1,2/2,2/3,2/4,2/5,2/6,
4/1,4/2,
5/1,5/2,5/3,5/4
}{
\draw [white,line width=3.0pt,double=black,double distance=3.0pt] (\x-0.5,\y) to (\x+0.5,\y);
\draw [white,line width=3.0pt,double=gray,double distance=3.0pt] (\x,\y-0.5) to (\x,\y+0.5);
}
\foreach \y in {1,3,5,7,9}
\draw [white,line width=3.0pt,double=black,double distance=3.0pt] 
(0.5,\y) to[in=90,out=180] (0,\y-0.5);
\foreach \y in {2,4,6,8}
\draw [white,line width=3.0pt,double=black,double distance=3.0pt] 
(0.5,\y) to[in=270,out=180] (0,\y+0.5);
\foreach \y in {1,3,5,7,9}
\draw [white,line width=3.0pt,double=black,double distance=3.0pt] 
(5.5,\y) to[in=270,out=0] (6,\y+0.5);
\foreach \y in {2,4,6,8}
\draw [white,line width=3.0pt,double=black,double distance=3.0pt] 
(5.5,\y) to[in=90,out=0] (6,\y-0.5);
\draw [white,line width=3.0pt,double=black,double distance=3.0pt] 
(6,9.5) to (6,10);
\draw [white,line width=3.0pt,double=black,double distance=3.0pt] 
(0,0.5) to (0,0);
\node[black] at (0,0) {\Huge \textbullet};
\node[black] at (-0.5,0) {\Large $s$};
\node[black] at (6,10) {\Huge \textbullet};
\node[black] at (6.5,10) {\Large $t$};
\foreach \x/\y in {1/9,2/7,3/1,4/3,5/5}
\node[black] at (\x+0.35,\y+0.35) {\small $c_{\x}$};
\foreach \x in {1,2,3,4,5}{ 
\draw[-|] (\x-0.5,-1.25) -- (\x,-1.25);
\draw[|-] (\x,-1.25) -- (\x+0.5,-1.25);
\node[black] at (\x,-1.75) {\x};
}
\foreach \y in {1,2,3,4,5,6,7,8,9}{ 
\draw[-|] (-1,\y-0.5) -- (-1,\y);
\draw[|-] (-1,\y) -- (-1,\y+0.5);
\node[black] at (-1.5,\y) {\y};
}
\pgfresetboundingbox \clip (-2,-2.5) rectangle (6,11);
\end{tikzpicture}}

\newcommand{\intermediateEDABC}{
\begin{tikzpicture}[thick,scale=1.25]
\foreach \x in {2,4,6}
\draw [white,line width=4.0pt,double=lightgray,double distance=4.0pt] 
(\x,9.5) to (\x,10) to[in=0,out=90] (\x-0.5,10.5);
\foreach \x in {1,3,5}
\draw [white,line width=4.0pt,double=lightgray,double distance=4.0pt] 
(\x,9.5) to (\x,10) to[in=180,out=90] (\x+0.5,10.5);
\foreach \x in {0,2,4}
\draw [white,line width=4.0pt,double=lightgray,double distance=4.0pt] 
(\x,0.5) to (\x,0) to[in=180,out=270] (\x+0.5,-0.5);
\foreach \x in {1,3,5}
\draw [white,line width=4.0pt,double=lightgray,double distance=4.0pt] 
(\x,0.5) to (\x,0) to[in=0,out=270] (\x-0.5,-0.5);
\foreach \x/\y in {1/9,2/7,3/1,4/3,5/5}{
\draw [gray,line width=2.0pt,double=lightgray,double distance=2.0pt] 
(\x,\y+0.25) to[in=0,out=90] (\x-0.25,\y+0.5) to (-0.5,\y+0.5);
}
\foreach \x/\y in {
2/8,2/9,
3/2,3/3,3/4,3/5,3/6,3/7,3/8,3/9,
4/4,4/5,4/6,4/7,4/8,4/9,
5/6,5/7,5/8,5/9
}{
\draw [white,line width=4.0pt,double=lightgray,double distance=4.0pt] (\x,\y-0.5) to (\x,\y+0.5);
\draw [white,line width=4.0pt,double=black,double distance=4.0pt] (\x-0.5,\y) to (\x+0.5,\y);
}
\foreach \x/\y in {1/9,4/3}{
\draw [white,line width=4.0pt,double=black,double distance=4.0pt] (\x-0.5,\y) to (\x+0.5,\y);
\draw [white,line width=4.0pt,double=lightgray,double distance=4.0pt] (\x,\y-0.5) to (\x,\y+0.5);
\draw [gray,line width=2.0pt,double=lightgray,double distance=2.0pt] (\x,\y-0.25) to (\x,\y+0.25);
}
\foreach \x/\y in {2/7,3/1,5/5}{
\draw [white,line width=4.0pt,double=lightgray,double distance=4.0pt] (\x,\y-0.5) to (\x,\y+0.5);
\draw [gray,line width=2.0pt,double=lightgray,double distance=2.0pt] (\x,\y-0.25) to (\x,\y+0.25);
\draw [white,line width=4.0pt,double=black,double distance=4.0pt] (\x-0.5,\y) to (\x+0.5,\y);
}
\foreach \x/\y in {
1/1,1/2,1/3,1/4,1/5,1/6,1/7,1/8,
2/1,2/2,2/3,2/4,2/5,2/6,
4/1,4/2,
5/1,5/2,5/3,5/4
}{
\draw [white,line width=4.0pt,double=black,double distance=4.0pt] (\x-0.5,\y) to (\x+0.5,\y);
\draw [white,line width=4.0pt,double=lightgray,double distance=4.0pt] (\x,\y-0.5) to (\x,\y+0.5);
}
\foreach \x/\y in {1/9,2/7,3/1,4/3,5/5}{
\draw [white,line width=3.0pt,double=white,double distance=3.0pt] 
(\x+0.25,\y-0.5) to (6.5,\y-0.5);
\draw [gray,line width=2.0pt,double=lightgray,double distance=2.0pt] 
(\x,\y-0.25) to[in=180,out=270] (\x+0.25,\y-0.5) to (6.5,\y-0.5);
\draw [gray,line width=2.0pt,double=lightgray,double distance=2.0pt] 
(\x,\y+0.25) to[in=0,out=90] (\x-0.25,\y+0.5);
}
\foreach \y in {1,3,5,7,9}
\draw [white,line width=4.0pt,double=black,double distance=4.0pt] 
(0.5,\y) to[in=90,out=180] (0,\y-0.5);
\foreach \y in {2,4,6,8}
\draw [white,line width=4.0pt,double=black,double distance=4.0pt] 
(0.5,\y) to[in=270,out=180] (0,\y+0.5);
\foreach \y in {1,3,5,7,9}
\draw [white,line width=4.0pt,double=black,double distance=4.0pt] 
(5.5,\y) to[in=270,out=0] (6,\y+0.5);
\foreach \y in {2,4,6,8}
\draw [white,line width=4.0pt,double=black,double distance=4.0pt] 
(5.5,\y) to[in=90,out=0] (6,\y-0.5);
\draw [white,line width=4.0pt,double=black,double distance=4.0pt] 
(6,9.5) to (6,10);
\draw [white,line width=4.0pt,double=black,double distance=4.0pt] 
(0,0.5) to (0,0);
\draw [gray,line width=2.0pt,double=lightgray,double distance=2.0pt] 
(0,0) -- (0,-0.5) to[in=180,out=270] (1,-1.5) -- (6.5,-1.5);
\draw [gray,line width=2.0pt,double=lightgray,double distance=2.0pt] 
(6,10) -- (6,10.5) to[in=0,out=90] (5,11.5) -- (-0.5,11.5);
\draw [gray,line width=2.0pt,double=lightgray,double distance=2.0pt] 
(-0.5,5.5) to[in=180,out=180] (-0.5,11.5);
\draw [gray,line width=2.0pt,double=lightgray,double distance=2.0pt] 
(-0.5,7.5) to[in=180,out=180] (-0.5,9.5);
\draw [gray,line width=2.0pt,double=lightgray,double distance=2.0pt] 
(-0.5,1.5) to[in=180,out=180] (-0.5,3.5);
\draw [gray,line width=2.0pt,double=lightgray,double distance=2.0pt] 
(6.5,8.5) to[in=0,out=0] (6.5,-1.5);
\draw [gray,line width=2.0pt,double=lightgray,double distance=2.0pt] 
(6.5,6.5) to[in=0,out=0] (6.5,0.5);
\draw [gray,line width=2.0pt,double=lightgray,double distance=2.0pt] 
(6.5,4.5) to[in=0,out=0] (6.5,2.5);
\node[black] at (0,0) {\Huge \textbullet};
\node[black] at (-0.3,0) {\LARGE $s$};
\node[black] at (6,10) {\Huge \textbullet};
\node[black] at (6.3,10) {\LARGE $t$};
\foreach \x/\y in {1/9,4/3}
\node[black] at (\x+0.3,\y+0.3) {\Large $c_{\x}$};
\foreach \x/\y in {2/7,3/1,5/5}
\node[black] at (\x-0.3,\y-0.3) {\Large $c_{\x}$};
\node[black] at (-0.35,0.7) {\Huge $\mathbf{H}$};
\node[gray] at (0.6,-0.1) {\Huge $\mathbf{V}$};
\node[darkgray] at (-0.45,-0.9) {\Huge $\mathbb{W}$};
\foreach \x in {0,1,2,3,4,5,6}{ 
\draw[-|] (\x-0.5,-2.25) -- (\x,-2.25);
\draw[|-] (\x,-2.25) -- (\x+0.5,-2.25);
\node[black] at (\x,-2.75) {\Large \x};
}
\foreach \x in {1,2,3}
\node[gray] at (7-\x-\x,-0.7) {\LARGE $\mathbf{+\x}$};
\foreach \x in {1,2,3}
\node[gray] at (\x+\x-1,10.7) {\LARGE $\mathbf{-\x}$};
\node[gray] at (-0.5,11.2) {\LARGE $\mathbf{-3}$};
\node[gray] at (-0.5,3.2) {\LARGE $\mathbf{-2}$};
\node[gray] at (-0.5,9.2) {\LARGE $\mathbf{-1}$};
\node[gray] at (6.4,-1.2) {\LARGE $\mathbf{+3}$};
\node[gray] at (6.4,0.8) {\LARGE $\mathbf{+2}$};
\node[gray] at (6.4,2.8) {\LARGE $\mathbf{+1}$};
\pgfresetboundingbox \clip (-1.5,-3.25) rectangle (8,11.75);
\end{tikzpicture}}

\newcommand{\meander}[1]{
\begin{tikzpicture}[scale=#1]
\draw[gray,line width=2.5pt] 
(-0.4,0) .. controls +(-225:0.4) and +(90:0.4) .. (-1,0);
\draw[black,line width=2.5pt] 
(-0.4,0) .. controls +(225:0.4) and +(-90:0.4) .. (-1,0);
\draw[black,line width=2.5pt] 
(0.4,0) .. controls +(-225:0.4) and +(45:0.4) .. (-0.4,0);
\draw[black,line width=2.5pt] 
(1,0) .. controls +(270:0.4) and +(-45:0.4) .. (0.4,0);
\draw[gray,line width=2.5pt] 
(1,0) .. controls +(-270:0.4) and +(45:0.4) .. (0.4,0);
\draw[gray,line width=2.5pt] 
(0.4,0) .. controls +(225:0.4) and +(-45:0.4) .. (-0.4,0);
\node[black] at (1,0) {\Large \textbullet};
\node[black] at (-1,0) {\Large \textbullet};
\node[black] at (1.5,0) {\Huge $\to$};
\node[black] at (-0.5,-0.4) {\Large $C$};
\pgfresetboundingbox \clip (-1,-1) rectangle (2,1);
\end{tikzpicture}}

\newcommand{\whitehead}[1]{
\begin{tikzpicture}[scale=#1]
\draw[white,line width=1.5pt,double=black,double distance=3.0pt] 
(-1.5,0) .. controls +(225:0.5) and +(-90:1) .. (-2.65,0);
\draw[white,line width=1.5pt,double=black,double distance=3.0pt] 
(1.5,0) .. controls +(-225:1.6) and +(45:1.6) .. (-1.5,0);
\draw[white,line width=1.5pt,double=gray,double distance=3.0pt] 
(-1,0) .. controls +(-225:1.2) and +(90:1.5) .. (-3,0);
\draw[white,line width=1.5pt,double=black,double distance=3.0pt] 
(-1,0) .. controls +(225:1.2) and +(-90:1.5) .. (-3,0);
\draw[white,line width=1.5pt,double=black,double distance=3.0pt] 
(1,0) .. controls +(-225:1) and +(45:1) .. (-1,0);
\draw[white,line width=1.5pt,double=black,double distance=3.0pt] 
(3,0) .. controls +(270:1.5) and +(-45:1.2) .. (1,0);
\draw[white,line width=1.5pt,double=black,double distance=3.0pt] 
(2.65,0) .. controls +(270:1) and +(-45:0.5) .. (1.5,0);
\draw[white,line width=1.5pt,double=gray,double distance=3.0pt] 
(2.65,0) .. controls +(-270:1) and +(45:0.5) .. (1.5,0);
\draw[white,line width=1.5pt,double=gray,double distance=3.0pt] 
(1.5,0) .. controls +(225:1.6) and +(-45:1.6) .. (-1.5,0);
\draw[white,line width=1.5pt,double=gray,double distance=3.0pt] 
(3,0) .. controls +(-270:1.5) and +(45:1.2) .. (1,0);
\draw[white,line width=1.5pt,double=gray,double distance=3.0pt] 
(-1.5,0) .. controls +(-225:0.5) and +(90:1) .. (-2.65,0);
\draw[white,line width=1.5pt,double=gray,double distance=3.0pt,shorten >=3.5pt] 
(1,0) .. controls +(225:1) and +(-45:1) .. (-1,0);
\node[black] at (3,0) {\Huge \textbullet};
\node[black] at (-3,0) {\Huge \textbullet};
\node[black] at (2.65,0) {\Huge \textbullet};
\node[black] at (-2.65,0) {\Huge \textbullet};
\pgfresetboundingbox \clip (-3,-1.25) rectangle (3,1.25);
\end{tikzpicture}}

\newcommand{\arc}[1]{
\begin{tikzpicture}[scale=#1]
\draw[darkgray,line width=3pt,densely dotted,->]
(-6,0) -- (6,0);
\node[darkgray] at (-6,-0.3) {\Large $A$};
\foreach \i in {1,2,3,4,5}{
\draw[black,line width=4pt] 
(\i,0) -- (\i,0.7);
\draw[gray,line width=2pt,double=lightgray,double distance=2pt] 
(\i,0) -- (\i,-0.7);
\node[black] at (\i,0) {\Huge \textbullet};
\draw[black,line width=4pt] 
(-\i,0) -- (-\i,-0.7);
\draw[gray,line width=2pt,double=lightgray,double distance=2pt] 
(-\i,0) -- (-\i,0.7);
\node[black] at (-\i,0) {\Huge \textbullet};
}
\node[black] at (-5,1) {\Large $R_1$};
\node[black] at (-5,-1) {\Large $B_1$};
\node[black] at (-4.65,-0.3) {\Large $s_1$};
\node[black] at (-4,1) {\Large $R_2$};
\node[black] at (-4,-1) {\Large $B_2$};
\node[black] at (-3.65,-0.3) {\Large $s_2$};
\node[black] at (-1,1) {\Large $R_k$};
\node[black] at (-1,-1) {\Large $B_k$};
\node[black] at (-0.65,-0.3) {\Large $s_k$};
\node[black] at (5,-1) {\Large $R_k$};
\node[black] at (5,1) {\Large $B_k$};
\node[black] at (5.35,-0.3) {\Large $t_k$};
\node[black] at (2,-1) {\Large $R_2$};
\node[black] at (2,1) {\Large $B_2$};
\node[black] at (2.35,-0.3) {\Large $t_2$};
\node[black] at (1,-1) {\Large $R_1$};
\node[black] at (1,1) {\Large $B_1$};
\node[black] at (1.35,-0.3) {\Large $t_1$};
\pgfresetboundingbox \clip (-6,-1.25) rectangle (6,1.25);
\end{tikzpicture}}

\newcommand{\parallelize}[1]{
\begin{tikzpicture}[scale=#1]
\draw[gray,line width=2pt,double=lightgray,double distance=2pt] 
(-3,-3.5) -- (-3,-2.5) to[out=90,in=-90] (1.5,1.5) -- (1.5,3.5) 
(-1,-3.5) -- (-1,-2.5) to[in=90,out=90] (5,-2.5) -- (5,-3.5)
(1,-3.5) -- (1,-2.25) to[in=90,out=90] (3,-2.25) -- (3,-3.5)
(-0.5,3.5) -- (-0.5,1) to[in=-90,out=-90] (-2.5,1) -- (-2.5,3.5);
\draw[darkgray,line width=2pt,densely dotted]
(-6,-3.5) -- (-6,3.5) 
(6,-3.5) -- (6,3.5);
\draw[black,line width=4pt]
(-6,-3) -- (6,-3) 
(-6,3) -- (6,3);
\node[black] at (-6.35,2.7) {\Large $s_2$};
\node[black] at (-6.35,-3.3) {\Large $s_1$};
\node[black] at (6.35,2.7) {\Large $t_2$};
\node[black] at (6.35,-3.3) {\Large $t_1$};
\draw[gray,line width=2pt,double=lightgray,double distance=2pt] 
(-6,3) -- (-6.5,3) 
(-6,-3) -- (-6.5,-3) 
(6,3) -- (6.5,3)
(6,-3) -- (6.5,-3);
\node[black] at (6,3) {\Huge \textbullet};
\node[black] at (6,-3) {\Huge \textbullet};
\node[black] at (-6,3) {\Huge \textbullet};
\node[black] at (-6,-3) {\Huge \textbullet};
\node[darkgray] at (-6.5,0) {\Large $S_{12}$};
\node[darkgray] at (6.5,0) {\Large $T_{12}$};
\node[black] at (-4.3,-3.4) {\Large $B_1$};
\node[black] at (-4.3,3.4) {\Large $B_2$};
\draw[black,line width=4pt,dashed]
(6,3) to[in=90,out=180] (5.5,2.5) -- (5.5,-2) to[out=270,in=0] (5,-2.5) -- (-4,-2.5) to[in=270,out=180] (-4.5,-2) -- (-4.5,1) to[in=180,out=90] (-4,1.5) -- (4.5,1.5) to[in=270,out=0] (5,2) -- (5,2.5) to[in=0,out=90] (4.5,3);
\draw[black,line width=4pt,dashed]
(-6,-3) to[in=270,out=0] (-5.5,-2.5) -- (-5.5,2) to[out=90,in=180] (-5,2.5) -- (4.375,2.5) to[in=0,out=0] (4.375,2) -- (-4.5,2) to[in=90,out=180] (-5,1.5) -- (-5,-2.5) to[in=180,out=270] (-4.5,-3);
\pgfresetboundingbox \clip (-6,-4) rectangle (6,4);
\end{tikzpicture}}

\newcommand{\straight}[1]{
\begin{tikzpicture}[scale=#1]
\draw[white,line width=2pt,double=lightgray,double distance=3pt]
(0,-0.25) to[out=270,in=180] (2.5,-0.75) -- (5.5,-0.75);
\draw[white,line width=2pt, double=lightgray, double distance=3pt]
(1,-0.25) to[out=270,in=180] (5.5,-2) to [out=0,in=270] (8,0) -- (8,4) to[out=90,in=0] (5.5,6) to[out=180,in=90] (0,4.25)
(2,-0.25) to[out=270,in=180] (5.5,-1.75) to [out=0,in=270] (7.75,0) -- (7.75,4) to[out=90,in=0] (5.5,5.75) to[out=180,in=90] (1,4.25)
(3,-0.25) to[out=270,in=180] (5.5,-1.5) to [out=0,in=270] (7.5,0) -- (7.5,4) to[out=90,in=0] (5.5,5.5) to[out=180,in=90] (2,4.25)
(4,-0.25) to[out=270,in=180] (5.5,-1.25) to [out=0,in=270] (7.25,0) -- (7.25,4) to[out=90,in=0] (5.5,5.25) to[out=180,in=90] (3,4.25)
(5,-0.25) to[out=270,in=180] (5.5,-1) to [out=0,in=270] (7,0) -- (7,4) to[out=90,in=0] (5.5,5) to[out=180,in=90] (4,4.25)
(5.5,-0.75) to [out=0,in=270] (6.75,0) -- (6.75,4) to[out=90,in=0] (5.5,4.75) to[out=180,in=90] (5,4.25);
\foreach \i in {1,2,3,4,5}
\draw[black,line width=3pt]
(\i,-0.25) -- (\i,4.25);
\draw[white,line width=2pt,double=black,double distance=3pt]
(0.0,-0.25) to[out=90,in=270] (0.0,0.0) to[out=90,in=270] (3.5,0.875) to[out=90,in=270] (2.5,1.25) to[out=90,in=270] (5.5,1.875) to[out=90,in=270] (0.5,2.75) to[out=90,in=270] (2.5,3.5) to[out=90,in=270] (0.0,4.0) to[out=90,in=270] (0.0,4.25);
\draw[white,line width=2pt,double=black,double distance=3pt]
(2,3.5) -- (2,0.0)
(1,3.0) -- (1,1.5)
(1,4.0) -- (1,3.5)
(3,1.5) -- (3,0.5)
(5,2.0) -- (5,0.0);
\pgfresetboundingbox \clip (-0.5,-2.5) rectangle (8.5,6.5);
\end{tikzpicture}}

\newcommand{\trace}[2]{
\begin{tikzpicture}[line width=1.5pt,black,scale=#2]
\tikzmath{\n = #1;}
\foreach \x [evaluate=\x as \y using {int(mod(\x,2))}] 
in {1,...,\n}{
\draw (0,2*\x+1) to[out=270,in=180] (0.5,2*\x+0.5) -- (\n+0.5,2*\x+0.5) to[out=0,in=90] (\n+1,2*\x) to[out=270,in=0] (\n+0.5,2*\x-0.5) -- (0.5,2*\x-0.5) to[out=180,in=90] (0,2*\x-1) (\x,2*\n+1) -- (\x,0) to[out=270,in=270] (1.5*\n-0.5*\x+2,0) -- (1.5*\n-0.5*\x+2,2*\n+1) to[out=90,in=90] (\x-1,2*\n+1);
}
\draw (0,1) to[out=270,in=180] (0.5,0.5) -- (\n+0.5,0.5) to[out=0,in=90] (\n+1,0) to[in=270,out=270] (\n+1.5,0) -- (\n+1.5,2*\n+1) to[out=90,in=90] (\n,2*\n+1);
\pgfresetboundingbox \clip (-0.5,-0.5*\n-0.5) rectangle (1.5*\n+2,2.5*\n+2);
\end{tikzpicture}}

\newcommand{\iton}[1]{
\begin{tikzpicture}[scale=#1]
\draw[line width=2pt,white,double=black,double distance=2pt] 
(7,1) -- (11,1) 
(7,2) -- (11,2) 
(7,3) -- (11,3)
(7,4) -- (11,4) 
(7,5) -- (11,5) 
(7,6) -- (11,6)
(7,7) -- (11,7);
\draw[line width=2pt,white,double=black,double distance=2pt]
(7,0) to[out=0,in=270]
(8,1) -- (8,7)
to[out=90,in=180] (9,8) to[out=0,in=90] (10,7) -- (10,1) to[out=270,in=180]
(11,0) to[out=0,in=270]
(12,1) -- (12,7)
to[out=90,in=180] (13,8);
\draw[line width=2pt,white,double=black,double distance=2pt]
(9,1) -- (13,1) 
(9,2) -- (13,2) 
(9,3) -- (13,3) 
(11,4) -- (13,4) 
(11,5) -- (13,5) 
(11,6) -- (13,6) 
(11,7) -- (13,7);
\draw[black,line width=3pt,->] 
(4,3.5)--(6,3.5);
\draw[line width=2pt,white,double=black,double distance=2pt] 
(1,1) -- (3,1) 
(1,2) -- (3,2) 
(1,3) -- (3,3);
\draw[line width=2pt,white,double=black,double distance=2pt]
(1,0) to[out=0,in=270]
(2,1) -- (2,7)
to[out=90,in=180] (3,8);
\draw[line width=2pt,white,double=black,double distance=2pt]
(1,4) -- (3,4) 
(1,5) -- (3,5) 
(1,6) -- (3,6) 
(1,7) -- (3,7);
\pgfresetboundingbox \clip (0,-1) rectangle (14,9);
\end{tikzpicture}}

\newcommand{\fence}[3]{
\begin{tikzpicture}[line width = 2.5pt,scale=#1]
\foreach \x in {0, 1, ..., #2}{
\foreach \y [evaluate=\y as \z using {int(mod(\x+\y,2))}] in {1,...,#3}{
\draw (\x,\y) .. controls +(30+120*\z:1) and +(-30-120*\z:1) .. (\x,\y+1);}}
\foreach \x in {2, 4, ..., #2}{
\draw (\x-2,1) .. controls +(-30:0.5) and +(210:0.5) .. (\x-1,1);
\draw (\x,#3+1) .. controls +(150:0.5) and +(30:0.5) .. (\x-1,#3+1);}
\draw (0,#3+1) .. controls +(150:0.5) and +(90:#3+2) .. (-0.4,0.5);
\draw (#2,1) .. controls +(-30:0.5) and +(-90:#3+2) .. (#2+0.4,#3+1.5);
\pgfresetboundingbox \clip (-1,0) rectangle (#2+1,#3+2);
\end{tikzpicture}}

\newcommand{\knit}[3]{
\begin{tikzpicture}[line width = 2.5pt,scale=#1]
\foreach \x in {0,1, ..., #2}{
\foreach \y [evaluate=\y as \z using {int(mod(\y,2))}] in {1,...,#3}{
\draw (\x+1.5*\z,\y) .. controls +(90:2) and +(-90:2) .. (\x+1.5-1.5*\z,\y+1);}}
\foreach \x in {2, 4, ..., #2}{
\draw (\x-0.5,1) .. controls +(-90:0.5) and +(-90:0.5) .. (\x+0.5,1);
\draw (\x,#3+1) .. controls +(90:0.5) and +(90:0.5) .. (\x-1,#3+1);}
\pgfresetboundingbox \clip (0,0.5) rectangle (#2+1.5,#3+1.5);
\end{tikzpicture}}

\newcommand{\net}[3]{
\begin{tikzpicture}[line width = 1.25pt,scale=#1]
\foreach \x in {0,1, ..., #2}{
\foreach \y [evaluate=\y as \z using {int(mod(\y+\x,2))}] in {1,...,#3}{
\draw (\x+1.5*\z,\y) .. controls +(-90:0.95) and +(90:0.95) .. (\x+1.5-1.5*\z,\y+1);}}
\foreach \x in {2, 4, ..., #2}{
\draw (\x-0.5,1) .. controls +(90:0.25) and +(90:0.25) .. (\x-1,1);
\draw (\x,#3+1) .. controls +(-90:0.25) and +(-90:0.25) .. (\x+0.5,#3+1);}
\pgfresetboundingbox \clip (0,0.5) rectangle (#2+1.5,#3+1.5);
\end{tikzpicture}}

\newcommand{\twistA}[1]{
\begin{tikzpicture}[scale=#1]
\foreach \x in {-0.4,-0.2,0,0.2,0.4} {
\draw [white,line width=0.5pt,double=gray,double distance=3.0pt] 
(0.875,\x-0.5) to[out=180,in=-90] (\x,0.25) -- (\x,2.5) .. controls +(0,0.5) and +(0,-0.5) ..  (-\x,4);
\draw[white,line width=0.5pt,double=gray,double distance=3.0pt] (\x,4) .. controls +(0,0.5) and +(0,-0.5) ..  (-\x,5.5) -- (-\x,7.25) to[out=90,in=0] (-0.875,8-\x);}
\foreach \y/\w in {0.5/1,1/1,2.25/1,6.5/1,7/1} {
\draw [white,line width=0pt,double=black,double distance=4.0pt] 
(-\w,\y) -- (\w,\y);
\foreach \x in {-0.4,-0.2,0,0.2,0.4} {
\draw [draw=gray, circle, fill=white] (\x,\y) circle (0.08); \node at (\x,\y) {\tiny $\ast$};}}
\pgfresetboundingbox \clip (1,-1) rectangle (1,9);
\end{tikzpicture}}

\newcommand{\twistB}[1]{
\begin{tikzpicture}[scale=#1]
\draw [white,line width=1.5pt,double=black,double distance=4.0pt] (0.65,4) to[in=0,out=0] 
(0.65,4.5) -- (-1,4.5);
\foreach \x in {-0.4,-0.2,0,0.2,0.4} {
\draw [white,line width=0.5pt,double=gray,double distance=3.0pt] 
(0.875,\x-0.5) to[out=180,in=-90] (\x,0.25) -- (\x,1.5) .. controls +(0,0.5) and +(0,-0.5) ..  (-\x,3) -- (-\x,4.75);
\draw[white,line width=0.5pt,double=gray,double distance=3.0pt] (\x,4.75) .. controls +(0,0.5) and +(0,-0.5) ..  (-\x,6.25) -- (-\x,7.25) to[out=90,in=0] (-0.875,8-\x);}
\foreach \y/\w in {0.5/1,1/1,4/0.65,6.5/1,7/1} {
\draw [white,line width=0pt,double=black,double distance=4.0pt] 
(-\w,\y) -- (\w,\y);
\foreach \x in {-0.4,-0.2,0,0.2,0.4} {
\draw [draw=gray, circle, fill=white] (\x,\y) circle (0.08); \node at (\x,\y) {\tiny $\ast$};}}
\draw [white,line width=1.5pt,double=black,double distance=4.0pt] (-0.65,4) to[in=180,out=180] 
(-0.65,3.5) -- (1,3.5);
\pgfresetboundingbox \clip (-1,-1) rectangle (1,9);
\end{tikzpicture}}

\newcommand{\twistC}[1]{
\begin{tikzpicture}[scale=#1]
\draw [white,line width=1.5pt,double=black,double distance=4.0pt] (0.65,5) to[in=0,out=0] 
(0.65,4.5) -- (-0.65,4.5);
\draw [white,line width=1.5pt,double=black,double distance=4.0pt] (0.65,5.5) to[in=0,out=0] 
(0.65,6) -- (-1,6);
\foreach \x in {-0.4,-0.2,0,0.2,0.4} {
\draw [white,line width=0.5pt,double=gray,double distance=3.0pt] 
(0.875,\x-0.5) to[out=180,in=-90] (\x,0.25) -- (\x,1) .. controls +(0,0.5) and +(0,-0.5) ..  (-\x,2.5);
\draw[white,line width=0.5pt,double=gray,double distance=3.0pt] (\x,2.5) .. controls +(0,0.5) and +(0,-0.5) ..  (-\x,4) -- (-\x,7.25) to[out=90,in=0] (-0.875,8-\x);}
\foreach \y/\w in {0.5/1,1/1,5/0.65,6.5/1,7/1} {
\draw [white,line width=0pt,double=black,double distance=4.0pt] 
(-\w,\y) -- (\w,\y);
\foreach \x in {-0.4,-0.2,0,0.2,0.4} {
\draw [draw=gray, circle, fill=white] (\x,\y) circle (0.08); \node at (\x,\y) {\tiny $\ast$};}}
\draw [white,line width=1.5pt,double=black,double distance=4.0pt] (-0.65,4.5) to[in=180,out=180] 
(-0.65,4) -- (1,4);
\draw [white,line width=1.5pt,double=black,double distance=4.0pt] (-0.65,5) to[in=180,out=180] 
(-0.65,5.5) -- (0.65,5.5);
\pgfresetboundingbox \clip (-1,-1) rectangle (1,9);
\end{tikzpicture}}

\newcommand{\twistD}[1]{
\begin{tikzpicture}[scale=#1]
\foreach \x in {-0.4,-0.2,0,0.2,0.4} {
\draw [white,line width=0.5pt,double=gray,double distance=3.0pt] 
(6.875,\x-0.5) to[out=180,in=-90]
(6+\x,0.25) -- (6+\x,1.5) .. controls +(0,0.5) and +(0,-0.5) ..  (6-\x,3);
\draw[white,line width=0.5pt,double=gray,double distance=3.0pt] (6+\x,3) .. controls +(0,0.5) and +(0,-0.5) ..  (6-\x,4.5);
\draw[white,line width=0.5pt,double=gray,double distance=3.0pt] (0-\x,5.5) --  (-\x,7.25) to[out=90,in=0] (-0.875,8-\x);
\draw[white,line width=0.5pt,double=gray,double distance=3.0pt] (3+\x,4.5) --  (3+\x,5.5);}
\foreach \y/\z in {0.5/6,1/6,5/3,6.5/0,7/0} {
\draw [white,line width=0pt,double=black,double distance=4.0pt] 
(-1,\y) -- (7,\y);
\foreach \x in {-0.4,-0.2,0,0.2,0.4} {
\draw [draw=gray, circle, fill=white] (\z+\x,\y) circle (0.08); \node at (\z+\x,\y) {\tiny $\ast$};}}
\foreach \x in {-0.4,-0.2,0,0.2,0.4} {
\draw [white,line width=0.5pt,double=gray,double distance=3.0pt] 
(6+\x,4.5) -- (6+\x,7.25) to[in=0,out=90] (5.25,8+\x) to[in=90,out=180] (4.5-\x,7.25) -- (4.5-\x,0.25) to[in=0,out=-90] (3.75,-0.5+\x) to[in=-90,out=180] (3+\x,0.25) -- (3+\x,4.5);
\draw [white,line width=0.5pt,double=gray,double distance=3.0pt] 
(3+\x,5.5) -- (3+\x,7.25) to[in=0,out=90] (2.25,8+\x) to[in=90,out=180] (1.5-\x,7.25) -- (1.5-\x,0.25) to[in=0,out=-90] (0.75,-0.5+\x) to[in=-90,out=180] (0+\x,0.25) -- (0+\x,5.5);}
\draw [white,line width=1.5pt,double=black,double distance=4.0pt] (5.25,5) -- (6.75,5);
\draw [white,line width=1.5pt,double=black,double distance=4.0pt] (0.75,5) -- (2.25,5);
\pgfresetboundingbox \clip (-1,-1) rectangle (5,9);
\end{tikzpicture}}

\newcommand{\maxloop}{
\tikz{
\draw[black,line width=1.2pt,double=white,double distance=1.2pt] (0,0.25) .. controls +(180:0.15) and +(180:0.4) .. (0.25,0);
\draw[black,line width=1.2pt,double=white,double distance=1.2pt] (-0.25,0) .. controls +(0:0.4) and +(0:0.15) .. (0,0.25);
\pgfresetboundingbox \clip (-0.25,0) rectangle (0.25,0.3);}}

\newcommand{\minloop}{
\tikz{
\draw[black,line width=1.2pt,double=white,double distance=1.2pt] (0.25,0.25) .. controls +(180:0.4) and +(180:0.15) .. (0,0);
\draw[black,line width=1.2pt,double=white,double distance=1.2pt] (0,0) .. controls +(0:0.15) and +(0:0.4) .. (-0.25,0.25);
\pgfresetboundingbox \clip (-0.25,0.05) rectangle (0.25,0.3);}}

\newcommand{\maxtwist}{
\tikz{
\foreach \x in {1,2,3}{
\draw[white,line width=1pt,double=black,double distance=1pt] (0.1*\x-0.9,0) to[out=90,in=180] (-0.5,0.4-0.1*\x) to[out=0,in=180] (0,0.1*\x);
\draw[white,line width=1pt,double=black,double distance=1pt] (0.5+0.1*\x,0) to[out=90,in=0] (0.5,0.1*\x) to[out=180,in=0] (0,0.4-0.1*\x);}
\pgfresetboundingbox \clip (-0.85,0.075) rectangle (0.85,0.3);}}

\newcommand{\mintwist}{
\tikz{
\foreach \x in {2,1,0}{
\draw[white,line width=1pt,double=black,double distance=1pt] (-0.1*\x-0.6,0.3) to[out=-90,in=180] (-0.5,0.2-0.1*\x) to[out=0,in=180] (0,0.1*\x);
\draw[white,line width=1pt,double=black,double distance=1pt] (0.8-0.1*\x,0.3) to[out=-90,in=0] (0.5,0.1*\x) to[out=180,in=0] (0,0.2-0.1*\x);}
\pgfresetboundingbox \clip (-0.85,0.075) rectangle (0.85,0.3);}}



\title{Universal Knot Diagrams}

\date{}

\author{
Chaim Even-Zohar
\and 
Joel Hass
\and 
Nati Linial
\and 
Tahl Nowik
}

\begin{document}
\begin{abstract}
We study collections of planar curves that yield diagrams for all knots. In particular, we show that a very special class called {\em potholder} curves carries all knots. This has implications for realizing all knots and links as special types of meanders and braids. We also introduce and apply a method to compare the efficiency of various classes of curves that represent all knots. 
\end{abstract} 

\maketitle

\vspace{-4.9ex}
\section{Introduction}

If a knot $K$ has a diagram whose underlying curve is~$C$ then we say that $K$ is {\em carried} by~$C$. In other words, some choice of the overcrossing and undercrossing arcs at every crossing point of~$C$ yields a diagram of~$K$. We are mainly concerned with infinite families of closed immersed planar curves $\mathcal{U} = \{C_1,C_2,\dots\}$ that eventually carry all knots. We say that $\mathcal{U}$ is {\em universal} if every knot is carried by all but finitely many curves in $\mathcal{U}$. 

The study of universal sequences of curves echoes a central theme in knot theory, of looking for simple or special ways to represent all knots. These include closed braids, grid diagrams, petal diagrams, bridge presentations, and many others.

This notion is also relevant to the evolving research of random knots. A~variety of random models have been proposed and studied recently \cite[see references therein]{even2017models}. One approach is to take a planar curve and choose every crossing independently at random. This produces a distribution over the resulting knot types. A~universal sequence of curves can be used in this fashion to construct distributions that eventually yield all knots. 

\begin{figure}[hb]
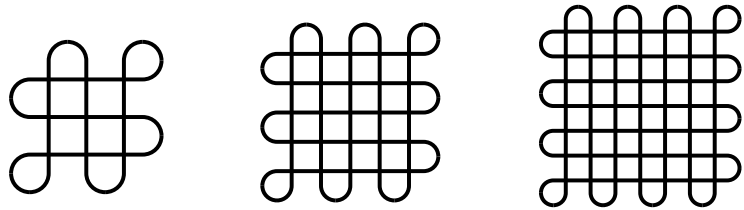

\centering
\potholder{3}{0.5} \;
\potholder{5}{0.39} \;\;\;
\potholder{7}{0.33}
\caption{Potholders}
\label{potholders}
\end{figure}

\bigskip\noindent
{\bf Examples of Universal Diagrams} 

\nopagebreak \medskip \noindent 
\textbf{(1)} The \emph{potholder} curve $P_n$ has $2n+1$ horizontal and $2n+1$ vertical line segments that are connected to form a closed immersed curve with $(2n+1)^2$ crossing points. The source for this name is that these curves resemble potholders used to handle hot pots, or to be placed under them. Figure~\ref{potholders} shows them for $n=1$, $2$, and $3$. 

Knots carried by potholder curves were studied by Grosberg and Nechaev \cite{grosberg1992algebraic, nechaev1996statistics}, under the name {\em lattice knots}. They studied the number of unknots carried by such curves via a connection to the Potts model of statistical mechanics. See also Adams et~al.~\cite{adams2017volume} who used the term potholder for similar curves, and Obeidin~\cite[Figure 14]{obeidin2016volumes}, following Garoufalidis and L{\^e}~\cite{garoufalidis2011asymptotics}. 

W.~Thurston suggested the possibility that diagrams similar to potholders are universal in a discussion of universality at {\it MathOverflow} \cite{38813}. Thurston wrote:
\begin{quotation}
I have a strong hunch that any knot type can be arranged in a form that has projection to one of these trajectories.
\end{quotation}  
The universality of potholders was also raised in~\cite[Problem 4]{burton2016computational}, as were other questions about universality of planar curves that are addressed in this paper. A main result in this work is that this family of curves is universal. Namely, every knot is carried by every large enough potholder. 

\begin{theorem}
\label{allpotholders}
All knots have a potholder diagram.
\end{theorem}

\smallskip
\noindent 
\textbf{(2)} The \emph{star} $S_n$ with $2n+1$ tips is a polygonal curve with $(2n+1)$ segments and $(n-1)(2n+1)$ crossing points, as illustrated in Figure~\ref{stars} for $n=2$, $3$, and~$4$. Star diagrams are special types of closed braids with $n$ strands, and they coincide with $(n,2n+1)$ torus knot projections~\cite{hayashi2012minimal, chang2015electrical}.

\begin{figure}[H]
\centering
\starlet{2}{1.6}
\starlet{3}{1.6}
\starlet{4}{1.6}
\caption{Stars}
\label{stars}
\end{figure}

Stars were shown to be universal by Adams et al.~in \cite{adams2015knot} and \cite[Section 4]{adams2015bounds}. See also~\cite{even2016invariants}.

\smallskip
\noindent 
\textbf{(3)} The trajectory of a billiard ball moving in the unit square with a slope $\pm p/q \in \Q$ forms the polygonal $p\,{:}\,q$ \emph{billiard curve}. The billiard curves of slope $5\,{:}\,2$, $7\,{:}\,3$, and~$9\,{:}\,4$ are depicted in Figure~\ref{billiards}.

\begin{figure}[H]
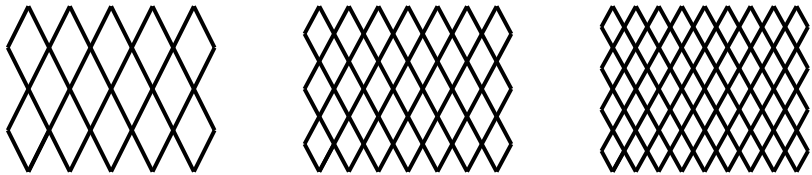

\centering
~ \\
\billiard{2}{2.75} \;\;\;
\billiard{3}{2.75} \;\;\;
\billiard{4}{2.75}
\caption{Billiard curves}
\label{billiards}
\end{figure}

The billiard curve is plane-isotopic to the well-studied \emph{Lissajous curve} or \emph{harmonic curve} in two dimensions~\cite{lissajous1857memoire,comstock1897real}. Knots carried by these curves or similar ones were studied by several researchers under the names \emph{billiard knots}, \emph{harmonic knots}, \emph{Lissajous knots}, \emph{Fourier knots}, \emph{Chebyshev knots}, and \emph{checkerboard knot diagrams}~\cite{kauffman1997fourier, jones1998lissajous, trautwein1995harmonic, buck1994random, bogle1994lissajous, lamm1997there, boocher2009sampling, koseleff2011chebyshev, cohen2015random, soret2016lissajous, rivin2016random}. 
  
Every knot is carried by a billiard curve, for some $p/q \in \Q$~\cite{koseleff2011chebyshev,lamm2012fourier}. In Figure~\ref{starbilliard} below, we show how every $(2n+1)$-star knot is carried by a \mbox{$(2n+1)\,{:}\,n$} billiard curve as in Figure~\ref{billiards}. Therefore, this sequence $B_n$ of billiard curves is also universal.

\smallskip
\noindent 
\textbf{(4)} Let $D_1, D_2, \dots$ be the underlying curves of minimum-crossing diagrams of all knot types, ordered in some way that is non-decreasing on crossing numbers. This sequence is far from being universal, since many knots are {\em not} carried by all but a finite subset of the curves. For example, consider the infinitely many diagrams that only carry two-bridge knots.

However, the sequence of connected sums of the first $n$ knot projections, $C_n = D_1 \# D_2 \# \,\cdots\, \# D_n$ is universal. Indeed, any knot is obtained from some $D_i$ with some choice of crossings, and the other crossings can always be chosen to make the remaining summands unknotted. We ignore the ambiguity in the definition of the connected sum of planar curves, as any choice for it would work. 

\medskip
These four examples illustrate several aspects of the notion of universality, and raise further questions. Universality is a qualitative property and quantitative aspects of it are of interest. Thus even though the last example is universal, it seems quite {\em inefficient} in carrying all knots. We will shortly return to this issue. 

Clearly a finite union of universal sets of curves is universal. However, each of the examples above seems as one {\em increasing} object, that perhaps even {\em converges}. Precise statements on such properties will be given below.

It is a major challenge to find simple properties that {\em characterize} universal curves. In  Sections~\ref{zigzag}-\ref{questions} we suggest some necessary and sufficient conditions, and examine further examples. We discuss open questions of this type and directions for future research.

\medskip
We assume throughout that the curves and knot diagrams we work with are regular and have transverse intersections in a finite number of double points.

\bigskip\noindent
{\bf Efficiency of Universal Diagrams}

\nopagebreak \medskip\noindent
Let $f:\N \to \R$. A set of curves $\mathcal{U}$ is \emph{$f$-efficient} if the curves with up to~$x$ crossings carry all knots with up to~$f(x)$ crossings. In other words, for every $x \in \N$, every knot~$K$ with up to $f(x)$ crossings is carried by some curve $C \in \mathcal{U}$ with at most~$x$ crossings. We say that $f$ is {\em the efficiency} of $\mathcal{U}$, if for every~$x \in \N$ the value of $f(x) \in \N$ is the largest for which the above condition holds. Note that this definition makes sense not only for a universal set, but for any set of curves that carry all knots. 

Let us demonstrate this notion of efficiency with the above examples of universal sets of curves.

\smallskip
\noindent 
\textbf{(1)} 
It is shown in~\cite{even2017distribution} that any $n$-crossing knot $K$ can be isotoped to a $(2n+1)$-petal diagram. It follows that $K$ is carried by a $(2n+1)$-star, with $\leq 2n^2$ crossing points. Thus the star curves $S_n$ are $\sqrt{x/2}$-efficient.

This efficiency is tight up to a multiplicative constant. This means that there exist $n$-crossing knots that are only carried by stars with $\Omega(n^2)$ crossings. One can show this, for example, using the additivity of the bridge index under the connected sum. In conclusion,

\begin{proposition*}
The efficiency of the star curves $S_n$ is $\Theta(\sqrt{x})$.
\end{proposition*}

\begin{figure}[b]
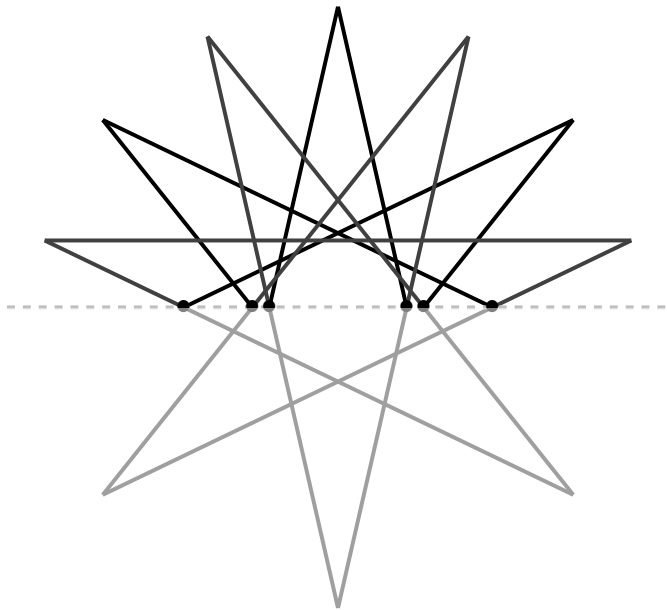

\centering
~ \\
\starbilliard{4.0}
\caption{Folding the $7$-star to a $7\,{:}\,3$ billiard curve, shown in the upper half plane. The $6$ points on the bisecting axis correspond to the $3$ left and $3$ right extremes of the billiard trajectory.}
\label{starbilliard}
\end{figure}

\smallskip
\noindent 
\textbf{(2)} 
It is possible to turn a $(2n+1)$-star diagram into a billiard diagram by folding across a bisecting axis as shown in Figure~\ref{starbilliard}. This yields a~\mbox{$(2n+1)\,{:}\,n$} billiard curve with $4n^2-n-1$ crossings. It is easy to see that a folded curve carries all knots that are carried by the original one. Therefore, these curves are $\sqrt{x/4}$-efficient. Tightness up to a constant follows similar to star curves.

\begin{proposition*}
The efficiency of the \mbox{$(2n+1)\,{:}\,n$} billiard curves $B_n$ is $\Theta(\sqrt{x})$. 
\end{proposition*}

This argument improves in general the efficiency that follows from previous proofs that billiard curves represent all knots. For example, Lamm's argument in~\cite{lamm2012fourier} yields a $p\,{:}\,q$ billiard diagram with $O(p^2r)$ crossings where $r$ is the length of a $p$-strand braid representation of the knot. 

Note also that the above construction has the feature that the ratio $p\,{:}\,q$ is bounded away from $0$ and $\infty$, and converges to $2$. We remark that there is nothing very special about the sequence \mbox{$(2n+1)\,{:}\,n$}. In fact, there are similar reductions from $p\,{:}\,q$ to $(p+q)\,{:}\,q$, or to $p\,{:}\,(mq)$ if $m$ is relatively prime to $2p$. One can use such arguments to show that the \mbox{$(3n+1)\,{:}\,3n$}, and therefore \mbox{$(n+1)\,{:}\,n$} billiard curves are universal as well.

\smallskip
\noindent 
\textbf{(3)} 
The efficiency of the cumulative connected sum $C_n$ defined above is much lower than that of stars and billiard curves. It is known that the number of knots with up to~$n$ crossings is exponential in~$n$. Thus this collection of curves is $\Omega(\log x)$-efficient.

This cannot be improved by more than a constant factor, as the first prime knot with $n$ crossings is not carried until all summands with less than $n$ crossings have occurred. 

\begin{proposition*}
The efficiency of the cumulative connected sums $C_n$ is $\Theta(\log x)$.
\end{proposition*}

\smallskip
\noindent 
\textbf{(4)}
Despite their superficial similarity with stars and billiard curves, the efficiency of potholders is presently unknown. 

\begin{question*}
How efficient are potholder diagrams?
\end{question*}

Out of the 84 knots with up to nine crossings, 82 are already carried by the \mbox{5-by-5} potholder curve~\cite{even2018diagram}. However, the reductions in our proofs can only show that potholders are $\Omega(\log x)$-efficient. See Corollary~\ref{meander-log}.

\medskip
More generally, it is natural to ask how efficient any universal sequence can be, whether it is $O(\sqrt{x})$, or~$\,x-\Omega(1)$, or anything in between. In Section~\ref{questions} we formulate this question precisely, and discuss some alternative approaches to universality and efficiency.

\bigskip \noindent 
{\bf Convergence of Universal Diagrams}

\nopagebreak \medskip \noindent
The set of planar curves is pre-ordered by the following relation. We write $C \leq C'$ if every knot that is carried by $C$ is also carried by $C'$. We say that a sequence of planar curves is \emph{increasing} if it is linearly ordered by this relation. This relation is dual in a sense to one studied by Kouki Taniyama, who defined a partial order on knots based on the curves that carry them~\cite{taniyama1989partial}.

One can consider another relation between curves, the partial order by inclusion. A~curve~$C$ is a \emph{subcurve} of~$C'$, denoted $C \subseteq C'$, if~$C$ is planar isotopic to a curve that is obtained from~$C'$ by one or more iterations of the following move. Cut $C'$ at two points adjacent to the same face, take one of the resulting two parts, and close it with a shortcut through that face. A~special case is when the two cutting points and the shortcut coincide at a crossing. For example $S_n \subseteq S_{n+1}$ as demonstrated by Figure~\ref{starincrease}.

\begin{figure}[b]
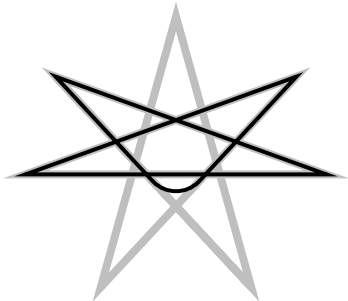

\centering
\starincrease{2}
\caption{The $5$-star is a subcurve of the $7$-star}
\label{starincrease}
\end{figure}

Observe that if $C \subseteq C'$ then $C \leq C'$. Indeed, if a knot is carried by~$C$, then it is also carried by~$C'$, with the same choice of crossings in the part that corresponds to the specified subcurve~$C$. The remaining part of~$C'$ can be made redundant if it is chosen to be decreasing and always passing under~$C$.

These relations let us refine our notion of a universal sequence. Recall the four examples mentioned in the introduction: potholder, stars, billiard curves, and cumulative connected sums. For each one of these sequences, one can verify that every curve is a subcurve of the subsequent one. The following proposition follows.

\begin{proposition*}
The universal sequences $P_n$, $S_n$, $B_n$ and $C_n$ are increasing.
\end{proposition*}

In some sense, this means that such a set of planar curves can be viewed as a single ``limit'' object. A more precise and slightly stronger formulation of this property is as follows. We say that a sequence of curves {\em converges} if there exists an infinite curve $C : \R \to \R^2$ such that the $n$th curve is the restriction of~$C$ to~$[-n,n]$. Note that the crossing points of this curve $C$ might not be a discrete set in the plane. The following proposition is easily verified, similarly to the previous one.

\begin{proposition*}
The universal sequences $P_n$, $S_n$, $B_n$ and $C_n$ converge.
\end{proposition*}

See~\cite{champanerkar2016geometrically} for another notion of convergence for knot diagrams.

\bigskip\noindent
{\bf Universality of Potholders and Meanders}

\nopagebreak \medskip\noindent
The {\em simple arc index} of an immersed curve is the minimal $k$ such that the curve can be split into $k$ simple arcs. An immersed curve with simple arc index two is called a {\em meander}. A knot diagram with simple arc index two is called a {\em meander diagram}. Figure~\ref{figureeight} shows such simple arcs in two different diagrams of the figure-eight knot.

\begin{figure}[b]
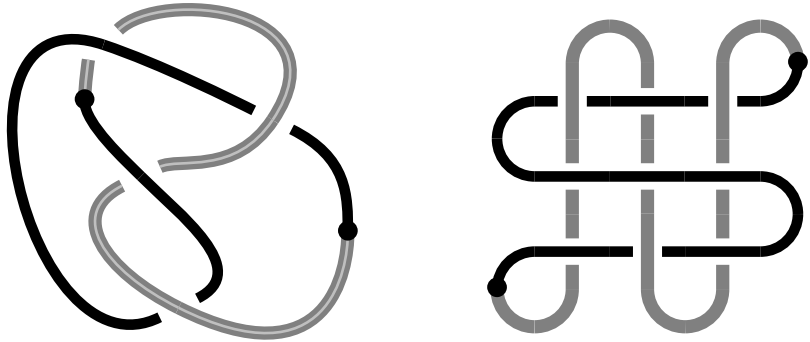

\centering
~ \\
\figureeight{1}
\potholderfigure{1}
\caption{Meander and potholder diagrams of the figure-eight knot}
\label{figureeight}
\end{figure}

The term ``meander'' is used elsewhere in mathematics to describe the intersection pattern of two self-avoiding planar curves, as these resemble a road crossing a meandering river~\cite{di2000folding, la2003approaches}. The use of this term in the context of knot and link diagrams originates in the work of Kappraff, Radovic and Jablan \cite{radovic2015meander, kappraff2016meanders}. 

It is natural to ask which knots can be represented by a meander diagram. The following result in this direction was established by von Hotz~\cite{von1960arkadenfadendarstellung} and by Ozawa~\cite{ozawa2007edge}. See also the more recent works,~\cite{adams2011complementary}, \cite{owad2018straight} and~\cite{belousov2018meander}. 

\nopagebreak
\begin{theorem}
\label{allmeanders}
All knots have a meander diagram.
\end{theorem}

Theorem~\ref{allmeanders} follows from a simple construction that starts with an arbitrary knot diagram and produces a meander diagram of the same knot. We provide a detailed proof in Section~\ref{meander} so that we can analyze and extend it. 

The second diagram in Figure~\ref{figureeight} demonstrates that potholder diagrams are a special class of meander diagrams. Indeed, one arc includes the horizontal segments and the other arc includes the vertical ones. The two change points are at the lower left and upper right corners. In Section~\ref{sec:potholder} we strengthen Theorem~\ref{allmeanders} and establish the universality of potholders.

\begin{theorem*}
[\bf \ref{allpotholders}]
All knots have a potholder diagram.
\end{theorem*}

\medskip
In Section~\ref{sec:link} we extend these results to links. A \emph{meander diagram for a $k$-component link} $L$ is a link diagram representing $L$ in which each component is decomposed into at most two simple arcs, which we can color black and red. Moreover all crossings occur between two different colors. See Figure~\ref{linkmeander} for meander diagrams of 3-component links.

\begin{figure}[b]
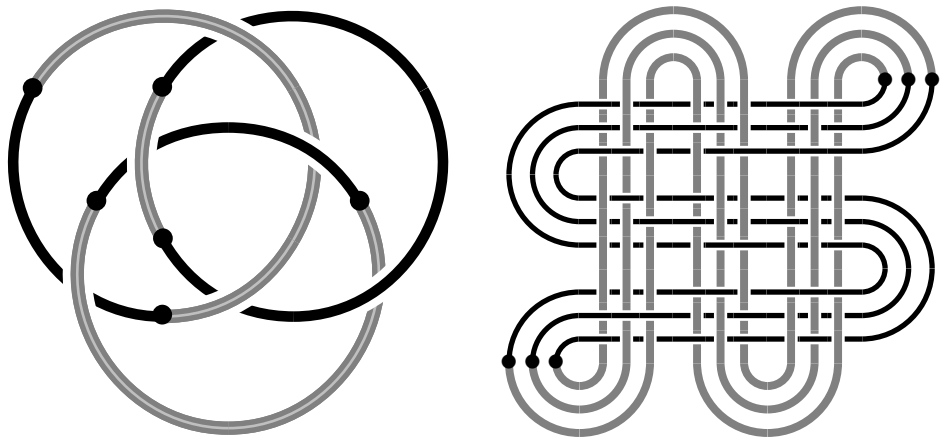

\centering
~ \\
~ \\
\borromean{2.5} \;\; \potholderlinkfigure{1.25}
\\ ~
\caption{Meander and potholder diagrams of 3-component links}
\label{linkmeander}
\end{figure}

Note that \cite{kappraff2016meanders} studied a different notion, in which a meander link consists of a pair of simple closed curves that intersect pairwise. That definition applies only to two-component links where each component is unknotted.  Ours is more general, as we show in the following theorem.

\begin{theorem}
\label{alllinksmeanders}
All links have a meander diagram.
\end{theorem}  

We also extend the notion of a potholder diagram to $k$-component links. The diagram on the right hand side of Figure~\ref{linkmeander} is a 3-by-3 potholder diagram of some random 3-component link. See Section~\ref{sec:link} for a precise definition. We prove the following theorem. 

\begin{theorem}
\label{alllinkspotholders}
All links have a potholder diagram.
\end{theorem}  

\medskip
In Section~\ref{sec:properties}
We discuss several consequences and applications. In particular, we show application to braids. A braid is called \emph{1-pure} if all strands except for one are parallel straight lines~\cite{artin1947theory}, as in Figure~\ref{straightbraid}.

\nopagebreak
\begin{corollary}
\label{StrongAlexander}
All knots are closures of $1$-pure braids.
\end{corollary}

\begin{figure}[t]
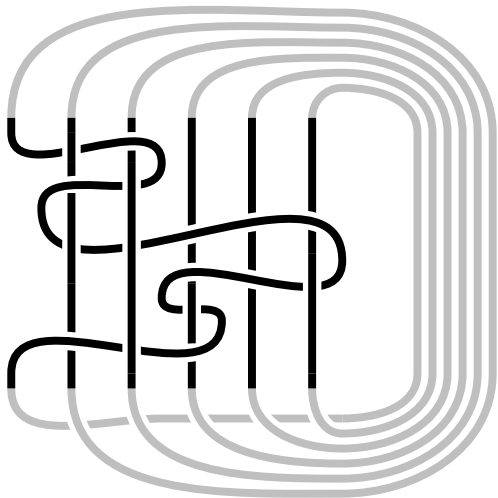

\centering
\straight{0.8}
\caption{A 1-pure braid with 6 strands [black], and the knot obtained by standard closure [gray].}
\label{straightbraid}
\end{figure}

Theorems~\ref{allpotholders} and \ref{alllinkspotholders} also answer a question by Adams, Shinjo and Tanaka \cite{adams2011complementary}, about the universality of diagrams with restricted edge numbers for the faces.

\begin{corollary}
\label{faces}
All links have a diagram with exactly two odd-sided faces.
\end{corollary}

\smallskip 
\noindent {\bf Remark.}
After submitting this article, we learned that Corollary~\ref{StrongAlexander} appears in a paper by G. Makanin from 1989~\cite{makanin1989analogue}. He uses the term unary braids, and shows their universality by different methods.  We thank Sergei Chmutov for bringing this to our attention.

\section{Meanders} \label{meander} 

We show that every knot can be represented by a meander diagram. We start by describing a procedure to turn a curve into a meander.

An \emph{$\mathfrak{X}$-move} at two successive crossings $c, c'$ is a regular homotopy of the curve supported in a neighborhood of the arc connecting $c, c'$ that creates a pair of new crossings as in Figure~\ref{xmove}. Note that this move always increases the number of crossings, and that we do not consider its inverse.

\begin{figure}[htbp]
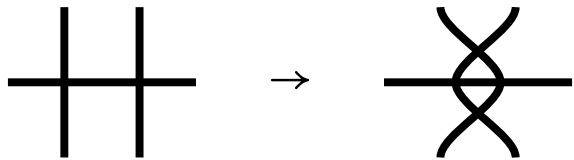

\centering
\xmove{1}
\caption{The $\mathfrak{X}$-move}
\label{xmove}
\end{figure}

\begin{lemma} \label{Xmoves}
Any closed curve can be transformed into a meander by a sequence of $\mathfrak{X}$-moves.
\end{lemma}

\nopagebreak
\begin{proof}
Let $C$ be an arbitrary curve with $n$ crossings. We start by choosing a base point and direction, and color the base point red and the rest of the curve black. As long as the next crossing in the black arc has type black-black, we extend the red arc, crossing the transverse black arc, so that the crossing point turns red-black. Otherwise, we will modify the curve by $\mathfrak{X}$-moves.  

Suppose that the type of the next crossing in the black arc is black-red. As we traverse the current black arc we encounter crossings $c_1, c_2, \dots, c_m$. If there are no black-black crossings we are done, because the curve is already a meander with the current black and red arcs.

Otherwise, let $c_k$ be the first black-black crossing. Since $c_1,\dots,c_{k-1}$ are all red-black, the black segment from $c_1$ to $c_k$ does not self-intersect and hence is embedded. See the left hand side of Figure~\ref{xmoves} for an example where $k=4$. In a neighborhood of this segment we perform a ($k-1$)-long sequence of $\mathfrak{X}$-moves at $(c_{k-1},c_k), \dots, (c_2,c_3), (c_1, c_2)$. We end up with a new configuration, as illustrated in the right hand side of Figure~\ref{xmoves}.

\begin{figure}[bh]
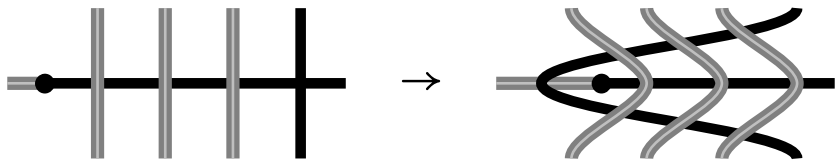

\centering
\xmoves{1}
\caption{Three $\mathfrak{X}$-moves}
\label{xmoves}
\end{figure}

Note that the new curve has $2(k-1)$ more crossing points. However, it has one fewer black-black crossing, and no red-red crossing point has been created. Therefore, by induction on the number of black-black crossings, the process terminates after at most $n$ such modifications with a curve that has only red-black crossings, i.e., a meander.
\end{proof}

Now we use this lemma about meander curves to deduce that meander knot diagrams can be used to represent all knots.

\begin{theorem*}[\bf\ref{allmeanders}]
All knots have meander diagrams.
\end{theorem*}

\begin{proof}
Let $K$ be a knot with an arbitrary diagram. By Lemma~\ref{Xmoves} there is a sequence of $\mathfrak{X}$-moves that transform the curve in the diagram of $K$ into a meander. We show that each $\mathfrak{X}$-move on a projection of $K$ can be lifted to an isotopy in $\R^3$ that preserves the knot type of $K$. 

\begin{figure}[hbt]
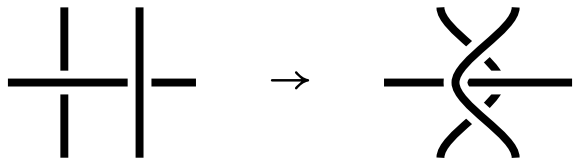

\centering
\xmovelifted{1}
\caption{A lifted $\mathfrak{X}$-move}
\label{Xmovelifted}
\end{figure}

Two successive crossings $c, c'$ give rise to four possible combinations of over-crossings or under-crossings. In each case an $\mathfrak{X}$-move lifts to an isotopy. One of the cases is shown in Figure~\ref{Xmovelifted}.
\end{proof}

We briefly discuss the implications of this argument for the efficiency of meanders. Given an $n$-crossing knot diagram, how many crossings are needed for an equivalent meander diagram? The above reduction yields an upper bound as follows.

The number of crossings in the meander is determined by the process described in Lemma~\ref{Xmoves}. Starting with $n$ black-black crossings, we turn the curve into a meander in $n$ steps. In every step, first the number of red-black crossings is at most tripled, and then one black-black crossing turns red-black. At the end all crossings are red-black. Solving the recurrence, the resulting meander has $(3^n-1)/2$ crossing points in the worst case. The following corollary summarizes this computation.

\nopagebreak
\begin{corollary}
\label{meander-log}
Meanders are at least $(\log_3 x)$-efficient.
\hfill $\square$
\end{corollary}

It is not clear if all knots can be represented by much smaller meanders than those provided by this procedure. A super-logarithmic efficiency for meanders would be interesting. 

However, it can be improved by a multiplicative constant to at least $\log_2 x$ as shown by Nicholas Owad in a recent preprint~\cite{owad2018straight}. He does so using the fact that the shortest distance of a self-crossing of the black arc to its ends is at most half its length. Thus the number of red-black crossings at most doubles, rather than triples. 

It is still easy to draw a knot diagram that leads to an exponential number of crossings by these algorithms. This is the case for a certain $3n$-crossing diagram of the connected sum of $n$ trefoils, while another $3n$-crossing diagram of the same knot is already a meander.

\medskip
We finish our discussion of meander knots by a slight strengthening of this notion.
We call the two points in a meander curve where it changes from one color to the other {\em change} points, and show that we can find a diagram for any knot in which these two points lie on the boundary of a single face. We call such a curve a {\em standard meander curve} and the corresponding diagram a {\em standard meander diagram}. 

After a planar isotopy, standard meander diagrams can be drawn as the union of an arc formed by the unit interval $I$ on the $y$-axis in the $xy$-plane and a simple arc connecting $(0,0)$ to $(0,1)$ that intersects the $y$-axis at a finite set of points on~$I$. See the left diagram on Figure~\ref{figMeanderPot}.

We strengthen Theorem~\ref{allmeanders} so that any knot is realized by this special type of meander diagram. See~\cite[Theorem 2.1]{adams2011complementary} for an earlier proof of this stronger version.

\begin{proposition}
\label{StandardMeander}
All knots have standard meander diagrams. 
\end{proposition}

\begin{proof}
We begin with an arbitrary meander diagram $D$ for a knot $K$. This consists of a pair of simple arcs $A$ and $B$,  running between the two change points $s$ and $t$. 

\begin{figure}[t]
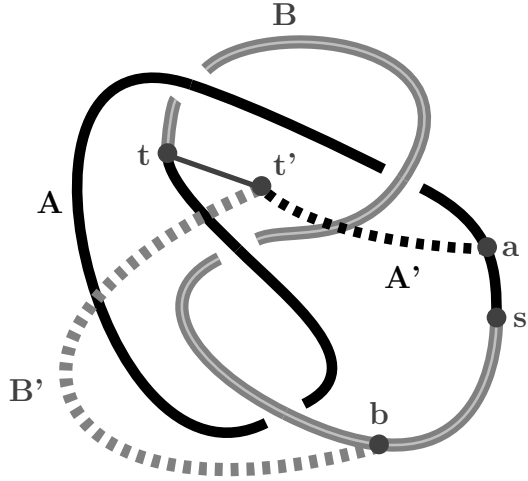

\centering
\standard{1.25}
\caption{Turning a meander diagram into a standard one}
\label{standard}
\end{figure}

Let $a$ be a point on $A$ close to $s$, $b$ a point on $B$ close to $s$, and $t'$ a point in a face adjacent to~$t$. Thus there exists a simple arc $tt'$ in the face running from $t$ to~$t'$. See Figure~\ref{standard}.

Neither $A$ nor $B$ separate the plane, so we can find simple arcs $A'$ and $B'$ running from $a$ to $t'$ and from $b$ to $t'$, with $A'$ disjoint from $A$ and $B'$ is disjoint from $B$. We can pick $A'$ and $B'$ to be disjoint from $tt'$ and also from the arc of $D$ between $a$ and $b$.

We now form a new diagram $D'$ by starting at $t$, and successively traversing $A, A', B', B$, so that the segment of $D$ between $a$ and~$b$ is replaced by the new arcs. The heights of strands of $K$ projecting to $A$ and $B$ are left unchanged. The height of~$A'$ is set to be larger than any point of $A$ and $B$ except for a vertical segment at its endpoint where it connects to $A$. The height of $B'$ is set to be greater still, above any point of $A$ and $B$ and $A'$ except for vertical segments at its endpoints where it connects to $B$ and $A'$. This choice of heights ensures that $D'$ also gives a diagram representing $K$. Moreover the diagram $D'$ is a standard meander diagram with arcs $A \cup A'$ and $B \cup B'$, with change points $t$ and $t'$ both meeting the face including the arc $tt'$.
\end{proof}

\section{Potholder Knots} 
\label{sec:potholder}

Recall, from the introduction to this paper, the definition of a potholder curve as in Figure~\ref{potholders}, and a potholder knot diagram, as in Figure~\ref{figureeight}. We now show that potholders are universal for knots.

\begin{theorem*}[\bf\ref{allpotholders}]
Any knot $K$ can be realized by a potholder knot diagram.
\end{theorem*}

\begin{proof}
The proof has two steps. We first take a standard meander diagram $D$ for the knot $K$, using Proposition~\ref{StandardMeander}. We next show how to construct a potholder diagram that is equivalent to $D$.

After an isotopy of $D$ on the 2-sphere, we can assume that the two change points are both adjacent to the unbounded region of the complement of $D$ in the plane. We also isotope $D$ so that its two simple arcs $A$ and $B$ are positioned as follows. Arc~$A$ is the vertical segment from $s = (0,0)$ to $t = (0,n+1)$ on the $y$-axis. Arc~$B$ is a simple arc that also runs from $s$ to $t$, crosses $A$ at the points $\{(0,1), (0,2), \dots (0,n)\}$ and is otherwise disjoint from the $y$-axis. See Figure~\ref{figMeanderPot}. We can assume that the number of crossings $n$ in $D$ is odd, by adding an extra loop using a Reidemeister type~I move at $t$. We also require that the intersection of $B$  with a small neighborhood of $s$ lies in the right half-plane. The latter condition can always be achieved by interchanging the two colors, if needed.

\begin{figure}[b]
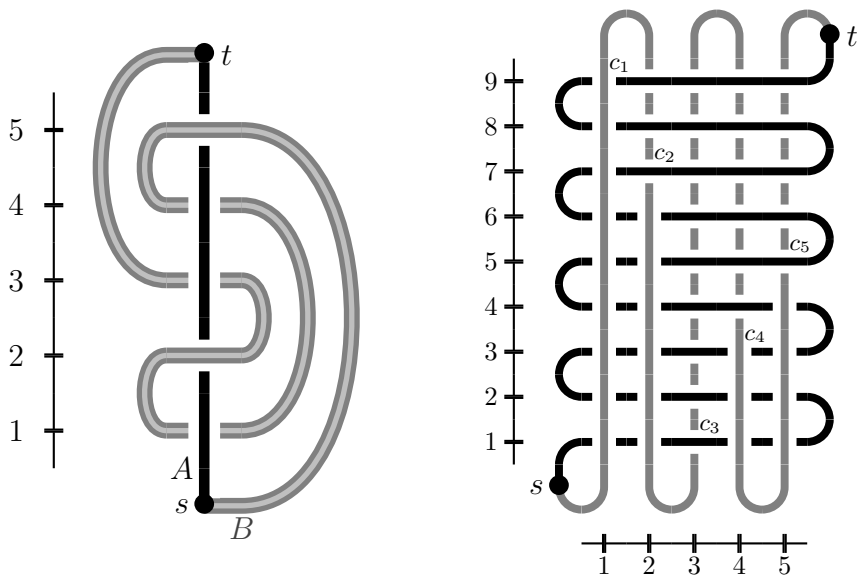

\centering
\meanderEDABC \;\;\;\;\;\;\;\; \potholderEDABC
\caption{A standard meander diagram and a corresponding potholder diagram of the figure-eight knot. The permutation is given by $\sigma:(1,2,3,4,5) \to (5,4,1,2,3)$.}
\label{figMeanderPot}
\end{figure}

Starting at $s$ and traversing $B$, we cross $A$ at the $y$-axis at $n$ points, in the order $(0,\sigma(1)), (0, \sigma(2)), \dots  , (0, \sigma(n))$ where $\sigma$ is a permutation of $\{1,\dots,n\}$. With our convention, $\sigma \in S_n $ uniquely determines the meander, although not all permutations give rise to a meander.

We next describe a  corresponding potholder diagram and later show that it defines a knot isotopic to $D$. We take $2n-1$ horizontal segments  on the lines $y = \{1, 2, \dots, 2n-1\}$ between $x=0$ and $x=n+1$, and $n$ vertical segments on the lines $x = \{1, 2, \dots, n\}$ between $y=0$ and $y=2n$. We connect these as in Figure~\ref{figMeanderPot} to form an $n \times (2n-1)$ rectangular potholder diagram $P$ with change points at $s=(0,0)$ and $t=(n+1,2n)$. One arc $H$ connects the change points and includes the horizontal segments, while $V$ also connects them and includes the vertical segments. We orient $H$ and $V$ so that they run from $s$ to $t$.

We now describe the choice of crossings for $P$. At each of the $n$ crossing points $c_i=(i, 2\sigma(i)-1)$ in~$P$, we let $H$ cross over~$V$ if $A$ crosses over~$B$ at $(0,\sigma(i))$ in~$D$, and cross under otherwise. At all other crossings $(i,j)$ in~$P$, choose $H$ to go under~$V$ if $j<2\sigma(i)-1$ and to go over~$V$ if $j>2\sigma(i)-1$. This causes the arc~$V$ to alternate between lying above~$H$ and lying below~$H$ as it passes through the crossings $c_1, c_2, \dots c_n$. See Figure~\ref{figMeanderPot}.

\begin{figure}[p]
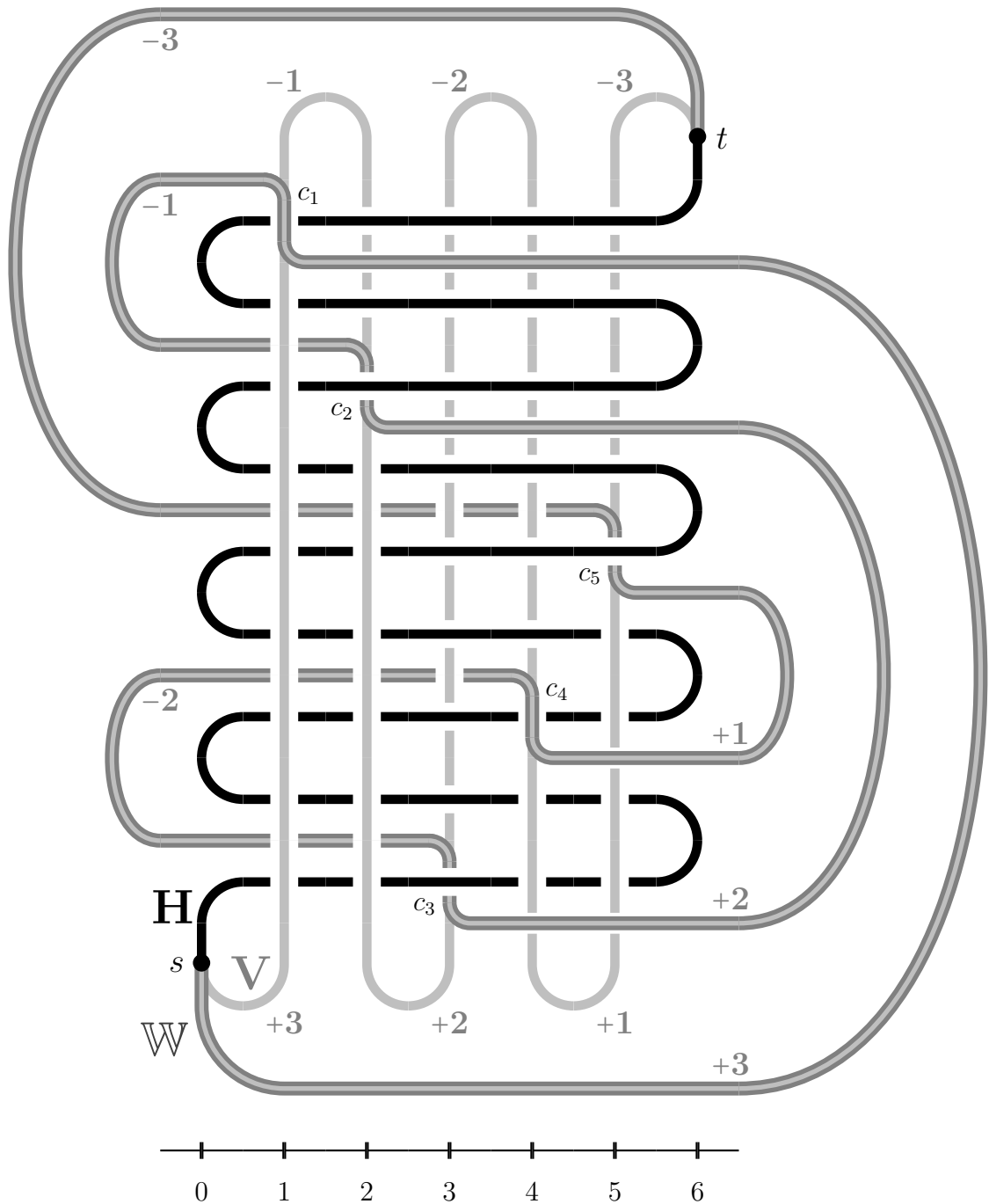

\centering
\intermediateEDABC
\caption{The meander diagram $M=H \cup W$ drawn over the potholder diagram $P= H \cup V$.
Numbers assigned to subarcs of $V$ and $W$ indicate their constant $z$-coordinates.}
\label{figMeanderToPotholder}
\end{figure}

In order to show that this potholder diagram realizes the correct knot, we introduce an intermediate knot diagram $M$, which is a meander knot diagram isotopic to $D$ and which also represents the same knot as $P$. One of the two meander arcs of $M$ coincides with $H$, the horizontal arc of $P$. The second meander arc $W$ runs between the two change points in $P$ as in Figure~\ref{figMeanderToPotholder}. 

The added arc $W$ starts at $s$, crosses~$H$ at the $n$ points $c_1,\dots,c_n$ in that order, and ends at $t$. Nearby to each crossing $c_i$ we continue the vertical arc of the crossing with two horizontal arcs, one with $y$-value just above the crossing, stretching left to $x=0$, and one with $y$-value just below the crossing, stretching right to $x=n+1$. We add an additional horizontal arc going to the left above $t$, and one going to the right below $s$. We then connect the $2n+2$ endpoints of these horizontal arcs with the same curved arcs as in the meander diagram $D$, after stretching these arcs so that their height increases by a factor of two. This completes the construction of~$W$.

Note that $D = A \cup B$ and $M = H \cup W$ are isotopic diagrams.  There is an isotopy of the plane that carries $H$ to the vertical arc $A$ and carries $W$ to $B$.  The construction of $M$ was chosen so that the crossings and the order in which they are reached agree in $D$ and $M$.

To finish the proof, we show that the diagram $M = H \cup W$ represents the same knot as $P = H \cup V$. We first denote by $W_i$ the subarc of $W$ between crossings $c_i$ and $c_{i+1}$, and by $V_i$ the subarc of $V$ between the same two crossings. This defines $W_0,\dots,W_n$ and $V_0,\dots,V_n$ where $c_0=s$ and $c_{n+1}=t$.

To construct an isotopy, we assign heights to each segment of $M$ and $P$. The arc $H$ will be fixed, and at height $z=0$ throughout the isotopy. We choose the $z$-value of $W$ to form a decreasing sequence of positive constant heights for the arcs $W_0, W_2, W_4,\dots,W_{n-1}$, and similarly a decreasing sequence of negative constant heights for $W_1, W_3, W_5, \dots, W_n$. Near each crossing $c_i$, these arcs connect with a vertical segment, chosen to be on the side of the crossing that gives the right choice of over- or under-crossing. Each subarc $V_i$ is assigned the same constant height as~$W_i$. See the heights indicated in Figure~\ref{figMeanderToPotholder}.
 
With this choice of heights, the arc $W_i$ is isotopic to the arc $V_i$ through an isotopy that keeps the height constant. The traces of these isotopies are at different heights and so are disjoint from one another. No isotopy crosses a vertical arc. Indeed, for $i$ even, the isotopy of the arcs $W_i$ and $V_i$ is contained in the half-space $x \ge i$ and lies above all crossings and all vertical arcs in this region. Similarly for $i$ odd, the isotopy of the arcs $W_i$ and $V_i$ is contained in $x \le i+1$ and lies below all crossings and all vertical arcs in that half-space.

In conclusion, the knot $K$ is represented by the $n \times (2n-1)$ potholder diagram~$P$. A square, $(2n-1) \times (2n-1)$ potholder diagram can be easily achieved, by adding redundant vertical arcs that over-cross all the horizontal ones.
\end{proof}

\section{Potholder Links}   
\label{sec:link}

In this section we show how the theory developed for knots extends to link diagrams. We show that all links, like knots, are carried by potholders.

\begin{theorem*}[\bf
\ref{alllinksmeanders}]
All links have meander diagrams.
\end{theorem*}

\begin{proof}
The proof follows the same lines as the argument for knots in Lemma~\ref{Xmoves}. One starts with a completely black diagram except for small red segments in each component. Then all black-black crossings are eliminated one by one, using sequences of $\mathfrak{X}$-moves wherever it is necessary.
\end{proof}

The following notion of parallel arcs is required for defining which meander link diagrams are standard. Two immersed arcs $A$ and $B$ are {\em parallel} if they are the boundary of an immersed small $\varepsilon$-tubular neighborhood of a single immersed arc~$C$, up to isotopy. For example, the two closed curves on the right hand side of Figure~\ref{standardmeanderlink} are parallel, and so are the two open arcs of each color. 

\begin{figure}[b]
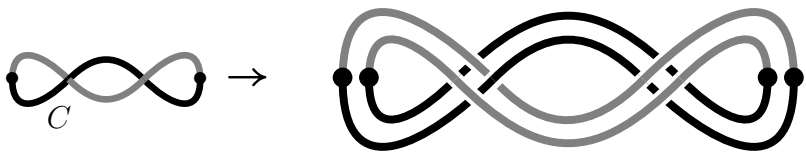
 
\centering
\meander{1.25}
\;\;\;
\whitehead{1}
\caption{A standard meander diagram of a 2-component link}
\label{standardmeanderlink}
\end{figure}

By taking $\varepsilon$ small enough, the two arcs are uniquely determined by~$C$, up to isotopy. We say in this setting that $A$ and $B$ are \emph{parallel copies of $C$}. Some simple but useful observations follow. Every intersection point between the original arc~$C$ and another arc~$D$, corresponds to exactly two intersection points between the two arcs and~$D$. A simple subarc of $C$ corresponds to a rectangular region, with subarcs of~$A$ and~$B$ as opposite sides. Every double point in~$C$ corresponds to four double points in $A \cup B$, at the corners of a square.

A collection of arcs $A_1, \dots, A_k$ is \emph{parallel} if $A_1$ is parallel to~$A_2$, $A_2$ is parallel to~$A_3$, and so on. Such arcs are also uniquely determined by a curve $C$. To be specific, take $A_1$ and~$B_1$ as two parallel copies of~$C$, then $A_2$ and~$B_2$ as two parallel copies of~$B_1$, and so on. An example of three parallel arcs appears in the potholder link diagram on the right hand side of Figure~\ref{linkmeander}.

We say that a $k$-component meander link diagram $D$ is {\em standard} if it is obtained from a standard meander curve 
$C$ by taking $k$ parallel copies of $C$ and making an arbitrary choice of overcrossing or undercrossing at each of the resulting crossings. 
One of the two simple arcs in $C$ is replaced by $k$ parallel arcs in $D$, which we can color red, 
and the other by $k$ parallel arcs which we color black. 
If $C$ has $n$ double points then $D$ has $nk^2$ crossings, each between different colors. See a standard meander diagram of the Whitehead link in Figure~\ref{standardmeanderlink}.

\begin{theorem}
\label{AllLinksStandardMeander}
All links have standard meander diagrams.
\end{theorem}

\begin{proof}
Let $D$ be a meander diagram for a link $L$ with $k$ components. Each component is composed of a black arc $B_i$ and a red arc $R_i$ for each $i \in \{1,\dots,k\}$. The two arcs forming each component meet at the change points $s_i$ and $t_i$. 
 
We start by showing that the change points can be arranged as in a standard meander diagram. This means that $s_1,\dots,s_k$ are successively adjacent to common faces, since they belong to parallel copies of a meander and correspond to its change point $s$. Similarly for $t_1,\dots,t_k$ which correspond to its other change point $t$. Since it should be a standard meander, with $s$ and $t$ adjacent to a common face, we show the same for $s_k$ and $t_1$. 

The following condition summarizes our requirement from the adjacency relations of the change points. There exists an oriented simple arc $A$ in the plane that meets the diagram transversely at the change points $s_1, s_2, \dots, s_k, t_1, t_2, \dots, t_k$ in that order, and is otherwise disjoint from $D$. Moreover, arcs of $D$ to the left of $A$ at $s_1, s_2, \dots, s_k$ are red, and arcs to the left at  $t_1, t_2, \dots, t_k$ are black. See Figure~\ref{arc}.

\begin{figure}[htbp]
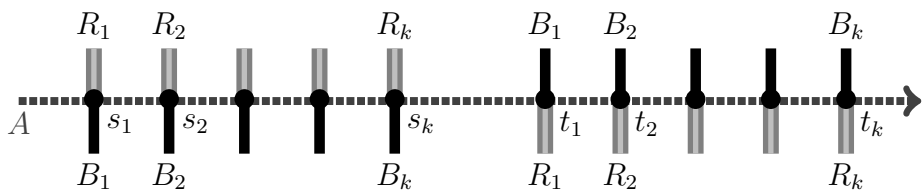
 
\centering
\arc{1}
\caption{The arc $A$ to be constructed}
\label{arc}
\end{figure}

We construct a meander diagram with such an arc $A$ for any link $L$. Let $D$ be any meander diagram of $L$ with change points $s_1',\dots,s_k'$ and $t_1',\dots,t_k'$. Start near~$s_1'$ and construct an arc~$A$ through~$s'_1$, with the red arc~$R_1$ to its left. Let $s_1 = s_1'$, and mark $2k-1$ successive points $s_2, \dots s_k, t_1, \dots t_k$ along that arc. We will change~$D$ to obtain a new meander link diagram with these change points, still representing~$L$.  

Cut the second component open near the given change point~$s_2'$, creating one endpoint on~$B_2$ and one on~$R_2$. Since the union of~$A$ and all the~$R_i$'s does not separate the plane, we can extend~$R_2$ to connect the endpoint of~$R_2$ to~$s_2$ from the left, without crossing a red arc or~$A$. For each new crossing point, where the red arc crosses a black arc, we choose the new red arc to go over the black arc. We then go back to the endpoint on~$B_2$ and notice that the black arcs together with~$A$ do not separate the plane, so we can extend~$B_2$ to meet~$s_2$ from the right, without crossing any black arcs. Here whenever a new crossing is created by the black arc meeting a red arc, we choose the black arc to go over the red.

The new diagram describes an equivalent link because it is obtained from the previous one by adding an unknotted loop starting and ending near $s_2'$ and lying completely above~$D$. This step is similar to the procedure described in the proof of Proposition~\ref{StandardMeander}, and illustrated in Figure~\ref{standard}.

Similarly, for each $i \in \{3,\dots,k\}$, the non-separating property of each color together with~$A$, allows us to successively cut the $i$th component near~$s_i'$ and extend the arcs~$R_i$ and~$B_i$ so that they meet at~$s_i$ with the red arcs to the left of~$A$ and the black arcs to the right. As before, we choose the new crossings to be over-crossings, so that we obtain an equivalent link diagram with no crossings between arcs of the same color.

We now deal with the remaining change points $ t_1', t_2', \dots, t_k'$. We again remove a small arc near~$t_1'$. The union of~$A$ and the red arcs still does not separate the plane. Thus the resulting endpoint of~$R_1$ near~$t_1'$ can be extended to meet~$t_1$ from the right, and without crossing any red arcs or~$A$. Similarly we can connect the black endpoint near~$t_1'$ to~$t_1$ from the left.

We have now created one loop in the union of~$A$ and the red arcs, because $A$ meets~$R_1$ in two points, and similarly in the union of~$A$ and the black arcs. So as we proceed to extend the subsequent arcs to $t_2,\dots,t_k$ we need to check that~$t_j'$ and both sides of~$t_j$ are in the same component of the plane after removing~$A$ and arcs of the same color.

Remove a small arc from~$R_2$ and~$B_2$ near $t_2'$. The loop~$O$ in $R_1 \cup A$ is formed by~$R_1$ and the part of~$A$ between~$s_1$ and~$t_1$. Orient the loop~$O$ to agree with the orientation of~$A$. The point~$t_2$ lies in the component to the left of~$O$, because of our choices for how~$R_1$ meets~$A$ at~$s_1$ and~$t_1$. The point~$t_2'$ also lies in the component to the left of~$O$, because all of~$R_2$ lies there. This component contains $R_2,R_3,\dots,R_k$ as disjoint arcs with one interior endpoint, and therefore these red arcs do not separate it. We can hence extend~$R_2$, this time from its end near~$t_2'$, to meet~$t_2$ from the right without introducing any crossings of red arcs. A similar argument applies to the arc~$B_2$ which can be extended to connect to~$t_2$ from the left.

We continue in this way, extending~$R_i$ and~$B_i$ to meet at~$t_i$ in the complement of the loops formed respectively by $A \cup R_{i-1}$ and by $A \cup B_{i-1}$. This gives a meander diagram with change points as in a standard diagram.

~

We now arrange for each collection of arcs $R_1,\dots,R_k$ and $B_1,\dots,B_k$ to be parallel as in a standard meander diagram. Put the midpoint of~$A$, a point lying between~$s_k$ and~$t_1$, at~$\infty$ and let $S_{12}$ and $T_{12}$ refer to the arcs of~$A$ between $s_1,s_2$ and $t_1,t_2$ respectively. Now consider the bounded region~$\Omega$, bounded by the arcs \mbox{$B_1 \cup B_2 \cup S_{12} \cup T_{12}$} which have disjoint interiors. All other black arcs $B_3 \dots B_k$ lie outside this region. Note that~$\Omega$ intersects $R_1, \dots, R_k$ in a potentially complicated pattern. We will isotope~$B_1$ and~$B_2$ in~$\Omega$ to give an equivalent link in which they will become parallel.

\begin{figure}[t]
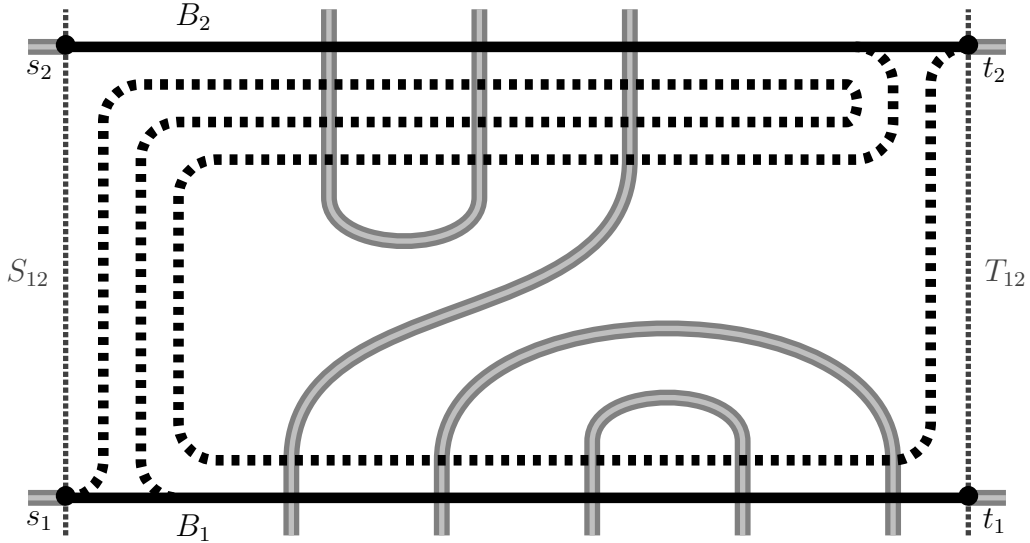

\centering
\parallelize{1}
\caption{making $B_1$ and $B_2$ parallel}
\label{parallelize}
\end{figure}

To do so, we proceed as in Figure~\ref{parallelize}. We cut~$B_1$ near~$s_1$. We extend~$B_1$ by going parallel to~$S_{12}$ until it almost reaches~$s_2$. It then runs parallel to~$B_2$ until it almost reaches~$t_2$, turns around and backtracks in parallel to rejoin~$B_1$ near~$s_1$. It crosses a red arc twice whenever~$B_2$ does, and we make each such crossing an overcrossing. Thus~$B_1$ is extended by pulling a small arc into a ``finger'' that lies above the rest of the link.

Now we cut~$B_2$ near~$t_2$. The extension of~$B_2$ starts by going parallel to~$T_{12}$ until just before it reaches~$t_1$. It then  turns around and follows~$B_1$ and its extension until it returns to the second point where~$B_2$ was cut. This extension of~$B_2$ goes around~$\Omega$ and crosses a red arc once whenever the original $B_1$ or~$B_2$ did. Again we make each such crossing an over-crossing. The resulting diagram gives an equivalent link, and the extended $B_1$ and~$B_2$ are now parallel.

We then isotope the arc~$B_3$ along with $B_1$ and~$B_2$ to get an equivalent meander link diagram in which all three are parallel. To do so we repeat the previous construction, but treating the parallel arcs $B_1, B_2$ as a single thickened arc. 

We repeat for all black arcs until we obtain an equivalent meander link diagram in which the~$k$ black arcs are parallel. We then repeat this construction for the red arcs $R_1,\dots,R_k$. In moving each red arc, we create new intersections with the black arcs but never destroy the property that the black arcs are parallel.

In conclusion, we have constructed an equivalent diagram of the given link~$L$ where the~$k$ components are parallel copies of a standard meander. This is a standard meander link diagram.
\end{proof}

The definition of a potholder diagrams of a link is similar to that of a standard meander. A~$k$-component \emph{potholder link diagram} is obtained by taking $k$ parallel copies of a potholder curve, and choosing the crossings arbitrarily. See the 3-component 3-by-3 potholder diagram in Figure~\ref{linkmeander}. We now show that such diagrams are universal for links.
 
\begin{theorem*}[\bf
\ref{alllinkspotholders}]
All links can be realized as potholder diagrams.
\end{theorem*}

\begin{proof}
Let $L$ be an arbitrary link. By Theorem~\ref{AllLinksStandardMeander} $L$~has a standard meander diagram. This diagram $D$ can be obtained by taking $k$ parallel copies of a standard meander curve~$C$, with~$n$ crossings, and making an appropriate choice for the~$k^2n$ resulting over and under crossings. The construction of the change points in the proof of Theorem~\ref{AllLinksStandardMeander} shows that we can take $n$ to be odd.

As in the proof of Theorem~\ref{allpotholders}, we transform~$D$ into a potholder diagram, treating the~$k$ parallel components of~$D$ as a single curve everywhere except at the~$k^2n$ crossings that correspond to $c_1,\dots,c_n$ in that proof. The potholder construction and the isotopies carry over to a band of~$k$ parallel components, but this band acquires some full twists due to Reidemeister~I moves. We next show how to eliminate those twists from the resulting link diagram, so that the $k$ strands remain parallel.

\begin{figure}[t]
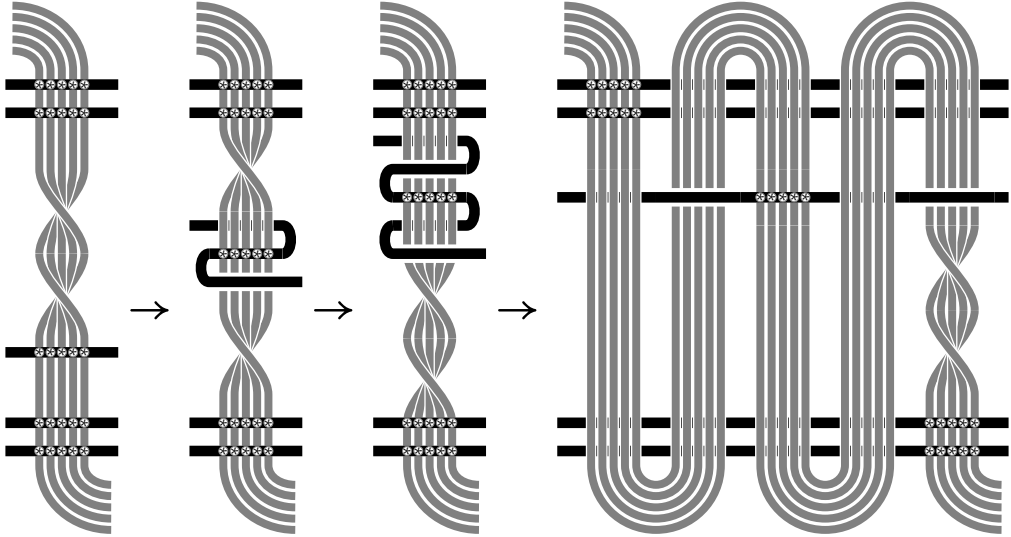

\centering
\twistA{0.75}
\raisebox{80pt}{\Huge $\to$}
\twistB{0.75}
\raisebox{80pt}{\Huge $\to$}
\twistC{0.75}
\raisebox{80pt}{\Huge $\to$}
\twistD{0.75}
\\ ~
\\ ~
\caption{Moving a twist in a vertical band of parallel strands across one horizontal strand while constructing a potholder link diagram. This operation can be applied with the initial crossings, those marked by small~$\circledast$'s, given arbitrarily.}
\label{twists}
\end{figure}

A closer look at the isotopy in the proof of Theorem~\ref{allpotholders} shows that Reidemeister~I moves are needed in two cases. Either they remove a negative loop,~\maxloop\; at some maximum point between two vertical segments of the potholder, or a positive loop,~\minloop\; at a minimum. As examples, see the pair of isotopic arcs labeled by $-1$ or the pair labeled by $+2$ in Figure~\ref{figMeanderToPotholder}. When applied to a band of $k$ strands, these moves create a negative twist at a maximum point,~\maxtwist\; and a positive twist at a minimum point,~\mintwist. Note that all the twists occur along the vertical bands. 

We now show how to transfer these twists along the underlying potholder curve, so that they cancel in pairs. Figure~\ref{twists} shows how one full twist in a vertical band can be moved across one horizontal strand. Using this move and its mirror images along the axes, one can transfer any number of twists across any horizontal strand along the way.

Finally, we note that there are equally many positive and negative full twists to cancel. With the notation of Figures~\ref{figMeanderPot}-\ref{figMeanderToPotholder}, observe that taking $W_i$ to $V_i$ introduces a twist if and only if $c_i$ occurs later than $c_{i+1}$ on $H$, or equivalently if $i$ is a \emph{descent} of $\sigma$ so that $\sigma(i) > \sigma(i+1)$. Since $W$ alternates between the two sides of~$H$, the twist is negative if the descent is odd, and positive if even. But every standard meander curve has an equal number of descents on each side, as can be verified by induction on $J$-moves,~$\parallel\,\,\leftrightarrow \mathlarger{\mathlarger{\between}}$.

\nopagebreak
We thus obtain a parallel $k$-component potholder link diagram representing the given link~$L$.
\end{proof}

\medskip
\section{Consequences} 
\label{sec:properties}

The fact that all knots and links have such simple representations as potholder diagrams is likely to find various applications in research. This is demonstrated in the current section.

\bigskip \noindent 
{\bf Closed 1-Pure Braids}

\nopagebreak \medskip \noindent
It was shown by Alexander~\cite{alexander1923lemma} that every knot can be realized by a closed braid. The universality of potholder knot diagrams implies a stronger version of this result, stating that all knots can be realized by a special class of braids.

Following Artin~\cite{artin1947theory}, we say that a braid on $k$ strands is {\em $1$-pure}, if $k-1$ strands of the braid are straight vertical arcs and the $k$th strand winds between them. See Figure~\ref{straightbraid}. Such braids form a normal subgroup of the pure braid group~$P_k$. The $1$-pure subgroup is isomorphic to a free group on $k-1$ generators, and the quotient isomorphic to~$P_{k-1}$. This decomposition plays an important role in computational problems on braids. 

Here we prove the following strong version of Alexander's Theorem. This uses the standard method of converting a pure braid into a knot. We connect the strands rotated by one, so that they form a single component.

\begin{corollary*}
[\bf \ref{StrongAlexander}]
All knots are closures of $1$-pure braids.
\end{corollary*}

\nopagebreak
\begin{proof}
Consider the sequence of {\em closed potholder braid} curves in Figure~\ref{toBraid}. Any knot that is carried by one of these curves is equivalent to the closure of a $1$-pure braid, and vice versa. Note that the bottom horizontal arc is introduced in passing from a braid to a closed knot. To put the bottom arc into a standard form, we perform $k-1$ Reidemeister~II moves, adding two horizontal arcs lying entirely over the vertical arcs they cross. In conclusion, it is sufficient to show that this sequence of curves is universal. 

\begin{figure}[bth]
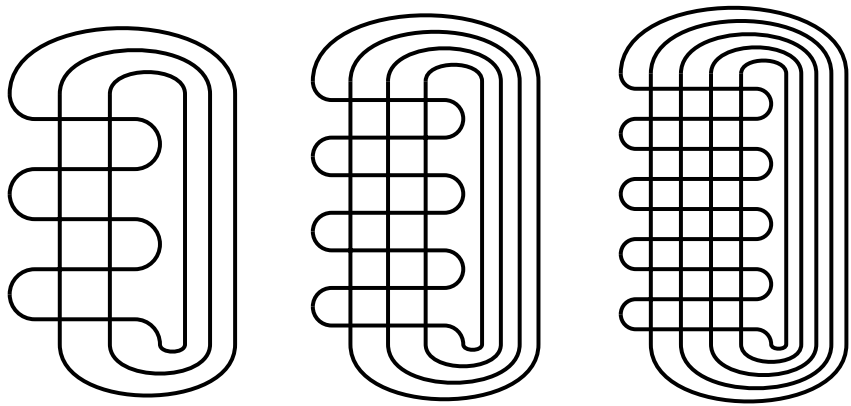

\centering
\trace{2}{0.666} \; \trace{3}{0.5} \;\;\; \trace{4}{0.4}
\caption{Potholder Braids}
\label{toBraid}
\end{figure}

Let~$K$ be an arbitrary knot. By Theorem~\ref{allpotholders}, $K$ can be represented by a potholder diagram. We say that an arc is {\em consistent} if all its crossings are over-crossings or all its crossings are under-crossings. We first construct a different, possibly larger, potholder diagram for $K$, in which the 1st, 3rd, 5th, and subsequent odd vertical segments are consistent.

To perform this transformation, we start with the construction of  Theorem~\ref{allpotholders}. Note that each  vertical strand traverses a sequence of over crossings followed by a sequence of under crossings. That is, there is a single change from over to under in each vertical strand. By adding a pair of additional vertical strands where needed, as in Figure~\ref{fig:ItoNmove}, we can transform the diagram to represent the same knot so that every other vertical strand is consistent.

\begin{figure}[tb]
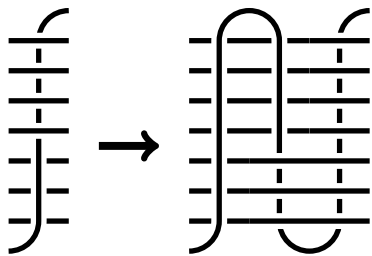
 
\centering
\iton{0.4}
\caption{An $I \to N$ move in a potholder diagram}
\label{fig:ItoNmove}
\end{figure}

All these vertical strands start and end next to the outer face of the diagram. Therefore, we can move all odd strands to the right and obtain a closed potholder braid diagram representing $K$, as in Figure~\ref{toBraid}.
\end{proof}

\noindent {\bf Remark.}
Regular potholders, as in Figure~\ref{potholders}, are already closed braids with another closure convention, sometimes referred to as the \emph{plat closure}. Thus it follows directly from Theorem~\ref{allpotholders} that also plat closures of 1-pure braids are universal for knots.

\bigskip \noindent
{\bf Even-Faces Diagrams}

\nopagebreak \medskip \noindent 
Adams, Shinjo and Tanaka study the combinatorial patterns arising from counting the number of sides of the complementary regions of a link diagram \cite{adams2011complementary}. They ask if links with $n$ components must have at least $n$ complementary regions with an odd number of sides. 

The realization of all links by potholder diagrams, stated in Theorem~\ref{alllinkspotholders}, gives a negative answer. 

\begin{corollary*}[\bf \ref{faces}]
Any $n$-component link has a diagram with exactly two odd-sided faces. Moreover, each of the two odd-sided faces can be specified to be a 1-gon  or a 3-gon.
\end{corollary*}

\begin{proof}
Every link has a potholder link diagram, and such diagrams have two odd-sided faces which are 1-gons. By a Reidemeister~I move, we can eliminate the 1-gon and turn a 4-gon into a 3-gon. All other faces maintain their parity.
\end{proof}

We note that Adams, Shinjo and Tanaka show in that paper how to embed an arbitrary link diagram inside a larger diagram whose faces are only 3-gons, 4-gons and 5-gons.  By following their construction starting with a potholder diagram, we obtain an explicit universal sequence with the same property.

\bigskip
\section{Zigzag Curves} 
\label{zigzag}

The {\em fence} curve is formed from an array of vertical strands, that go up and down while zigzagging right and left as in Figure~\ref{fence}. The open curves in this figure and the subsequent ones are assumed to be closed through the outer face. 

\begin{figure}[b]
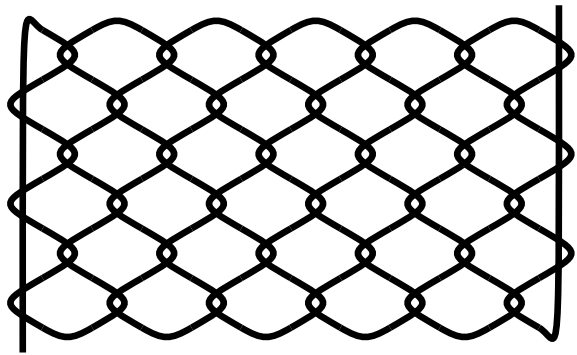

\centering
\fence{0.66}{10}{6}
\caption{The fence curve}
\label{fence}
\end{figure}

Consider a sequence of such curves, where both the width and the height go to infinity, such as the sequence of $n \times n$ fences. We find the following question very intriguing.

\begin{question*} 
Are fence curves universal?
\end{question*}

At first glance, fences appear similar to billiard and potholder curves, being composed of a simple periodic pattern that repeats in all cardinal directions. A~closer look suggests that fences might be more restrictive. Indeed, all crossings occur between two points at a short distance along the curve.

A more concrete manifestation of this restriction will be given in terms of the following definitions. The {\em Gauss word} of a planar curve is obtained by listing the crossings in their order of occurrence along the curve. For example, the word $abcdbadc$ corresponds to the figure-eight diagram in Figure~\ref{figureeight} on the left. Of course, this word is defined up to a cyclic ordering or renaming the points. A {\em subword} is obtained by omitting some of the crossings.

It is easy to see that a knot diagram whose Gauss word doesn't contain the subword $abab$ must represent the unknot. A diagram that avoids the subword $abcacb$ only yields connected sums of $(2,p)$ torus knots. One can verify that the fence curve never contains $abcabc$. This property may similarly reflect an obstruction to its universality.

Here is another form of our question. Can all knots be generated as plat closures of braids generated by the squares of the standard generators? This subgroup of braids was considered in \cite{pride1986tits,collins1994relations}. Their plat closures coincide with knots carried by a fence curve.

\medskip
We sampled some random knots by picking the over- and under-crossings in the fence curve independently at random. Using the software {\em SnapPy}~\cite{culler999snappy}, we identified the resulting knot types, see Appendix~\ref{experimental}. We were able to find quite easily fence representations for every knot with up to 8 crossings. However, eleven 9-crossing knots did not appear as fence diagrams. In Rolfsen's notation, these knots are $9_{29}$, $9_{34}$, $9_{35}$, $9_{37}$, $9_{38}$, $9_{39}$, $9_{41}$, $9_{46}$, $9_{47}$, $9_{48}$, and $9_{49}$. All other 9-crossing knots did occur, each appearing at least 500 times. This was based on over 150 million samples, with 50 to 200 crossings. The source code and results are available at~\cite{even2018diagram}.

As yet we do not have a convincing explanation for this phenomenon. Either, for some reason, these knots occur much less frequently than their 9-crossing counterparts, or they are not realized by fences. In either case it would be interesting to find a knot invariant, or a measure of complexity, that explains the observed difference.

\begin{figure}[tb]
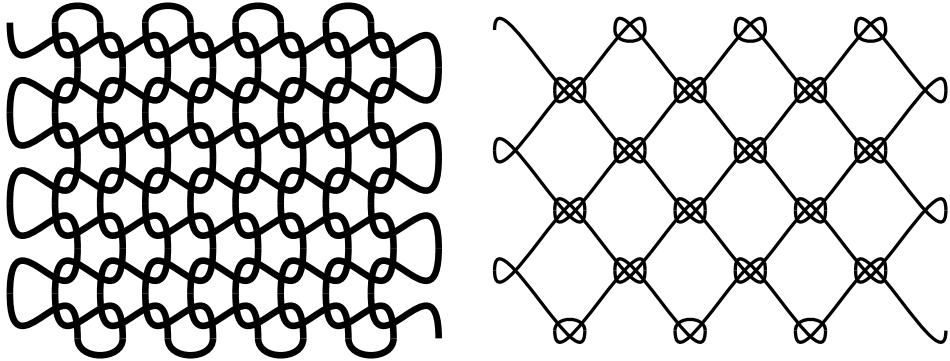

\centering
~ \\
~ \\
\knit{0.6}{8}{7} \;\;\;\;
\net{0.8}{6}{5}
\\ ~
\caption{The knit and the net}
\label{knitnet}
\end{figure}

\medskip
Patterns that are similar to the fence are found in knitted fabrics, fishing nets, and more. See Figure~\ref{knitnet}. 
One can verify that the knitted curve on the left is universal if and only if the fence is. The net curve on the right also avoids some subwords, such as $(abcdefg)^2$.

\bigskip 
\section{Questions} 
\label{questions}

This study of universal knot diagrams raises many additional questions aside from those mentioned above. Here we focus on two main directions for future research: the attempts to characterize universality, and its trade-off with efficiency.

\bigskip \noindent
{\bf Efficiency} 

\nopagebreak \medskip \noindent
As we saw in the introduction, the sequence $S_n$ of $(2n+1)$-star diagrams is universal, increasing, and $\sqrt{x/2}$-efficient. We wonder whether a universal increasing sequence of curves can have efficiency asymptotically bigger than $\sqrt{x}$. Concretely we ask:
\begin{question}
\label{f}
For how large $f:\N\to\R$ does there exist a universal, increasing, and $f(x)$-efficient sequence of curves?
\end{question}

Notions of efficiency make sense also for general families of curves. For example, the sequence~$D_n$ of all knot projections is trivially $x$-efficient, but is neither universal nor increasing. A crucial difference between these two families is that whereas $D_n$ has exponentially many curves of each size, $S_n$~has only one. Let us say that a set of curves has a {\em growth} function $g:\N\to\R$, if it has at most $g(x)$ curves with up to~$x$ crossings. The above discussion suggests the following question.

\begin{question}
\label{fg}
For what functions $f,g:\N \to \R$ does there exist an $f(x)$-efficient set of curves with growth $g(x)$, that represents all knots? Specifically,
\begin{itemize}
\item How small can the growth be, for an $\Omega(x)$-efficient set?
\nopagebreak
\item How high can the efficiency be, if the growth is $O(\text{poly}(x))$?
\end{itemize}
\end{question}

We stress a major difference between these two problems. Question~\ref{fg} applies to general sets of curves, not necessarily universal or increasing. In Question~\ref{f} growth is not a consideration, and we require instead that the universal set is increasing, or possibly, even more strongly, that it converges to one infinite curve. The following construction explains why.

Consider all the curves that are obtained as the connected sums of the $(2n+1)$-star $S_n$ and a knot projection with up to $2n^2$ crossings. This family is universal thanks to the large stars, but it is not increasing because of the other summand. It is $x/2$-efficient, because it contains every knot projection together with a star of comparable size. Clearly, universality and efficiency are combined here in an artificial way, and indeed this family has an exponential growth rate.

\medskip
The efficiency of universal sequences of curves, and the various relations between them give rise to a notion of complexity classes of curves.

\nopagebreak
Recall that we write $C \leq C'$ if every knot carried by~$C$ is also carried by~$C'$. Similarly for two sets of curves, let $\mathcal{U} \leq \mathcal{U}'$ if every knot carried by some~$C \in \mathcal{U}$ is also carried by some~$C' \in \mathcal{U}'$. Many of our arguments in this paper can be described in terms of this relation.

We say that $\mathcal{U}$ {\em reduces} to $\mathcal{U}'$ if $\mathcal{U} \leq \mathcal{U}'$ and this relation holds with polynomial efficiency. Namely, there is a polynomial $p(\cdot)$, so that every knot that is carried by some~$C \in \mathcal{U}$ with $n$ crossings is carried by some~$C' \in \mathcal{U}'$ with at most $p(n)$ crossings. 

Thus we can talk about complexity classes. For example, stars and billiard curves are in the same complexity class as the set of all curves. The cumulative connected sum has a higher complexity, since its efficiency is $\Theta(\log x)$. Meanders and potholders are in the same complexity class, but which class this is remains presently unknown.

\bigskip \noindent
{\bf Characterization}

\nopagebreak \medskip \noindent
Which sequences of planar curves are universal? We hope to find criteria that are based on simple properties of the curves.

In order to find necessary conditions for universality, one can consider some knot invariants. Let the {\em width} of a knot be the minimum, over its 3-space conformations, of its largest number of intersection points with a horizontal plane. Let a {\em height function} for a planar curve be any smooth real function on the plane with no critical point on a crossing. Let the {\em width} of a planar curve be the minimum, over height functions, of its largest number of intersection points with a level set. 

Observe that the width of a knot is at most the width of a curve that carries it. Indeed, given a height function, the curve can be appropriately lifted to realize any choice of crossings. As we show in Appendix~\ref{largewidth}, there are knots of arbitrarily large width. Therefore,
\nopagebreak
\begin{proposition*}
Every universal sequence of curves has widths that diverge to $\infty$.
\end{proposition*}

 This notion of curves of large width seems related to planar graphs without very small separators in the sense of~\cite{lipton1979separator}. We presently don't know if any such condition yields universality. The $n \times n$ fence curve would be a useful test case. 

\medskip
There are various natural properties of planar curves which suggest themselves as sufficient conditions for universality. In particular one wonders what it is that underlies the universality of concrete sequences such as the potholders. For example, an $n \times n$ {\em weave} is the union of $n$ horizontal and $n$ vertical strands, such as ${\equiv}\!\!{|}\!\!{\equiv}\!\!{|}\!\!{\equiv}\!\!{|}\!\!{\equiv}$. In a key step at the proof of Theorem~\ref{allpotholders} we embed a meander knot in a large weave. It is conceivable that this is a special case of a more general phenomenon. Specifically we ask,

\begin{question}
Let $\mathcal{U}$ be a set of curves. Suppose that for every $n \in \N$, all but finitely many curves of $\mathcal{U}$ contain an $n \times n$ weave by some restriction to a disk. Is~$\mathcal{U}$ necessarily universal?
\end{question}
 
A positive answer would imply at once the universality of potholders, stars, billiard curves, and more. Still, it is not hard to see that containing a weave is not a necessary condition for universality.

Here is another potential extension of Theorem~\ref{allpotholders}. Recall the definition of Gauss words from Section~\ref{zigzag}. Call a word {\em bipartite} if every letter appears once in its first half and once in its second half. Note that a curve is a meander if and only if it has a bipartite Gauss word. Hence every knot has a diagram that avoids non-bipartite subwords.

One way to think of our proof is as embedding bipartite words as subwords in large enough potholders. This view suggests the following question.

\begin{question}
Let $\mathcal{U}$ be a set of curves. Suppose that every bipartite word is a subword of all but finitely many of the curves. Is $\mathcal{U}$ necessarily universal?
\end{question}

As far as we know, this condition may also be {\em necessary} for universality. To show this direction one could construct knots, all of whose diagrams must contain a given bipartite word. The question whether that is impossible seems very interesting on its own. Is there a bipartite subword that can be avoided by a suitable diagram of every knot?

\bigskip

\medskip
\appendix
\section{Large Width}
\label{largewidth}

\begin{proposition*}
The width of the $(p,q)$ torus knot is $2\min(p,q)$.
\end{proposition*}

\begin{proof}
The standard realization of the $(p,q)$ torus knot on a torus of revolution has $p$ minima and $p$ maxima in the direction of the axis of revolution. Hence it intersects any plane perpendicular to the axis in at most $2p$ points. Similarly the width is at most~$2q$. We show that this is tight.

Consider any smooth curve in $\R^3$ that realizes the $(p,q)$ torus knot. The curve lies on an embedded torus~$T$. After a slight perturbation, the height function on~$T$ can be taken to be a Morse function. Intersect $T$ with a fixed collection $H$ of horizontal planes, one between each pair of critical levels. These planes split $T$ into finitely many pieces that are glued along simple closed curves. Each piece is homeomorphic to a sphere with 1, 2, or 3 disks removed. 

Suppose that an innermost curve on one of these planes bounds a disk $E$ in~$T$. Then we eliminate this curve by an isotopy of $T$ in $S^3$, which fixes $T$ away from a small neighborhood of~$E$, and isotops~$E$ to lie on one side of the plane. This is possible since the two disks form a 2-sphere which bounds a ball on each side, one of which is disjoint from the rest of~$T$. This move eliminates at least one piece of~$T$ in the complement of the planes~$H$, and fills in a disk in the boundary of another piece. It preserves the property that each piece is a sphere with up to three disks removed.

These isotopies of $T$ in $S^3$ cannot make it disjoint from all planes in~$H$, as the pieces of $T$ in the complement of $H$ always have genus zero, while $T$ is a torus. Hence at some point we have a closed curve in $T \cap H$ that does bound a disk in the complement of~$T$, but doesn't bound a disk in~$T$. Such a closed curve is necessarily homotopic to either the meridian or the longitude of the torus. 

In conclusion, the given knotted curve lies on an embedded torus, that contains either a meridian or a longitude at a constant height. The knot intersects this meridian or longitude transversely at least $p$ or $q$ times, with the same orientation. Somewhere on that plane, we must have as many intersection points with the opposite orientation. Together there are at least $2\min(p,q)$ intersections at that height.
\end{proof}

\noindent {\bf Remark.}
An unbounded width for torus knots also follows from the work of Kuiper~\cite{kuiper1987new}. He studied a similar knot invariant $d(K)$ where one considers the largest number of intersections with \emph{any} transversal plane, rather than only horizontal ones. 
Ozawa~\cite{ozawa2010waist} studied these invariants and named them ``trunk'' and ``super-trunk''. His results together with the bridge numbers of torus knots also imply the above proposition.

\section{Fences}
\label{experimental}

The following table gives the count of prime knots occurrences in random fence diagrams, as identified by SnapPy~\cite{culler999snappy, even2018diagram}. See Section~\ref{zigzag}. The knots are denoted here by their Dowker--Thistlethwaite names.

\medskip
{ \small
\begin{multicols}{5}
\noindent
\texttt{K3a1	117898 \\
K4a1	330729 \\
K5a1	237325 \\
K5a2	47640 \\
K6a1	71746 \\
K6a2	144965 \\
K6a3	237764 \\
K7a1	72112 \\
K7a2	143917 \\
K7a3	44581 \\
K7a4	72321 \\
K7a5	37727 \\
K7a6	30722 \\
K7a7	7550 \\
K8a1	45124 \\
K8a2	8062 \\
K8a3	6645 \\
K8a4	44901 \\
K8a5	72382 \\
K8a6	20701 \\
K8a7	37816 \\
K8a8	20806 \\
K8a9	37376 \\
K8a10	44997 \\
K8a11	72592 \\
K8a12	576 \\
K8a13	3347 \\
K8a14	2417 \\
K8a15	2363 \\
K8a16	10525 \\
K8a17	37611 \\
K8a18	31239 \\
K8n1	16180 \\
K8n2	8113 \\
K8n3	3428 \\
K9a1	6824 \\
K9a2	6718 \\
K9a3	44811 \\
K9a4	16257 \\
K9a5	3298 \\
K9a6	4821 \\
K9a7	6762 \\
K9a8	44641 \\
K9a9	6584 \\
K9a10	44501 \\
K9a11	4822 \\
K9a12	20880 \\
K9a13	10400 \\
K9a14	10266 \\
K9a15	20670 \\
K9a16	5934 \\
K9a17	37954 \\
K9a18   0 \\
K9a19	20580 \\
K9a20	20740 \\
K9a21	37868 \\
K9a22	37649 \\
K9a23	5958 \\
K9a24	12398 \\
K9a25	916 \\
K9a26	11639 \\
K9a27	21922 \\
K9a28	0 \\
K9a29	0 \\
K9a30	0 \\
K9a31	0 \\
K9a32	0 \\
K9a33	4837 \\
K9a34	8680 \\
K9a35	9099 \\
K9a36	14551 \\
K9a37	622 \\
K9a38	4374 \\
K9a39	3839 \\
K9a40	0 \\
K9a41	830 \\
K9n1	16114 \\
K9n2	16241 \\
K9n3	6675 \\
K9n4	16017 \\
K9n5	0 \\
K9n6	0 \\
K9n7	0 \\
K9n8	0 \\
K10a*	542580 \\
K10n*	79705 \\
K11a*	546338 \\
K11n*	138160 \\
K12a*	555864 \\
K12n*	204192 \\
K13a*	540127 \\
K13n*	280692 \\
K14a*	523955 \\
K14n*	386622 \\
others	~155606675
}
\end{multicols}
}

\medskip
\section*{Acknowledgements}

\medskip\noindent\nopagebreak
This work was supported by BSF Grant 2012188. The numerical experiments were carried out with the facilities of the  School of Computer Science and Engineering at HUJI, supported by
ERC 339096. Work of J. Hass was also supported by NSF grant DMS-1758107.

\medskip
\bibliographystyle{alpha}
\bibliography{universal}

\newcommand{\etalchar}[1]{$^{#1}$}
\begin{thebibliography}{ACSF{\etalchar{+}}15}

\bibitem[ACD{\etalchar{+}}15]{adams2015knot}
Colin Adams, Thomas Crawford, Benjamin DeMeo, Michael Landry, Alex~Tong Lin,
  MurphyKate Montee, Seojung Park, Saraswathi Venkatesh, and Farrah Yhee.
\newblock Knot projections with a single multi-crossing.
\newblock {\em Journal of Knot Theory and Its Ramifications}, 24(03):1550011,
  2015.

\bibitem[ACSF{\etalchar{+}}15]{adams2015bounds}
Colin Adams, Orsola Capovilla-Searle, Jesse Freeman, Daniel Irvine, Samantha
  Petti, Daniel Vitek, Ashley Weber, and Sicong Zhang.
\newblock Bounds on {\"u}bercrossing and petal numbers for knots.
\newblock {\em Journal of Knot Theory and Its Ramifications}, 24(02):1550012,
  2015.

\bibitem[AKC{\etalchar{+}}17]{adams2017volume}
Colin Adams, Alexander Kastner, Aaron Calderon, Xinyi Jiang, Gregory Kehne,
  Nathaniel Mayer, and Mia Smith.
\newblock Volume and determinant densities of hyperbolic rational links.
\newblock {\em Journal of Knot Theory and Its Ramifications}, 26(01):1750002,
  2017.

\bibitem[Ale23]{alexander1923lemma}
James~Wadell Alexander.
\newblock A lemma on systems of knotted curves.
\newblock {\em Proceedings of the National Academy of Sciences}, 9(3):93--95,
  1923.

\bibitem[Art47]{artin1947theory}
Emil Artin.
\newblock Theory of braids.
\newblock {\em Annals of Mathematics}, pages 101--126, 1947.

\bibitem[AST11]{adams2011complementary}
Colin Adams, Reiko Shinjo, and Kokoro Tanaka.
\newblock Complementary regions of knot and link diagrams.
\newblock {\em Annals of Combinatorics}, 15(4):549--563, 2011.

\bibitem[BDHZ09]{boocher2009sampling}
Adam Boocher, Jay Daigle, Jim Hoste, and Wenjing Zheng.
\newblock Sampling {L}issajous and {F}ourier knots.
\newblock {\em Experimental Mathematics}, 18(4):481--497, 2009.

\bibitem[BEET16]{burton2016computational}
Benjamin Burton, Herbert Edelsbrunner, Jeff Erickson, and Stephan Tillmann.
\newblock Computational geometric and algebraic topology.
\newblock {\em Oberwolfach Reports}, 12(4):2637--2699, 2016.

\bibitem[BHJS94]{bogle1994lissajous}
MGV Bogle, JE~Hearst, VFR Jones, and L~Stoilov.
\newblock Lissajous knots.
\newblock {\em Journal of Knot Theory and its Ramifications}, 3(02):121--140,
  1994.

\bibitem[BM18]{belousov2018meander}
Yury Belousov and Andrei Malyutin.
\newblock Meander diagrams of knots, links, and tangles: proofs of
  {J}ablan--{R}adovi{\'c} conjectures.
\newblock {\em preprint arXiv:1803.10879}, 2018.

\bibitem[Buc94]{buck1994random}
Gregory~R Buck.
\newblock Random knots and energy: Elementary considerations.
\newblock {\em Journal of Knot Theory and its Ramifications}, 3(03):355--363,
  1994.

\bibitem[CDW]{culler999snappy}
Marc Culler, Nathan~M Dunfield, and Jeffrey~R Weeks.
\newblock Snap{P}y, a computer program for studying the geometry and topology
  of 3-manifolds.

\bibitem[CE15]{chang2015electrical}
Hsien-Chih Chang and Jeff Erickson.
\newblock Electrical reduction, homotopy moves, and defect.
\newblock {\em preprint arXiv: 1510.00571}, 2015.

\bibitem[CK15]{cohen2015random}
Moshe Cohen and Sunder~Ram Krishnan.
\newblock Random knots using {C}hebyshev billiard table diagrams.
\newblock {\em Topology and its Applications}, 194:4--21, 2015.

\bibitem[CKP16]{champanerkar2016geometrically}
Abhijit Champanerkar, Ilya Kofman, and Jessica~S Purcell.
\newblock Geometrically and diagrammatically maximal knots.
\newblock {\em Journal of the London Mathematical Society}, 94(3):883--908,
  2016.

\bibitem[Col94]{collins1994relations}
Donald~J Collins.
\newblock Relations among the squares of the generators of the braid group.
\newblock {\em Inventiones Mathematicae}, 117(1):525--529, 1994.

\bibitem[Com97]{comstock1897real}
Elting~H Comstock.
\newblock The real singularities of harmonic curves of three frequencies.
\newblock {\em Trans. of the Wisconsin Academy of Sciences}, 11:452--464, 1897.

\bibitem[DF00]{di2000folding}
Philippe Di~Francesco.
\newblock Folding and coloring problems in mathematics and physics.
\newblock {\em Bulletin of the American Mathematical Society}, 37(3):251--307,
  2000.

\bibitem[EHLN]{even2017distribution}
Chaim Even{-Zohar}, Joel Hass, Nati Linial, and Tahl Nowik.
\newblock The distribution of knots in the petaluma model.
\newblock {\em Algebraic \& Geometric Topology}, to appear.

\bibitem[EHLN16]{even2016invariants}
Chaim Even{-Zohar}, Joel Hass, Nati Linial, and Tahl Nowik.
\newblock Invariants of random knots and links.
\newblock {\em Discrete \& Computational Geometry}, 56(2):274--314, 2016.

\bibitem[EZ]{even2018diagram}
Chaim Even-Zohar.
\newblock Diagram sampler. {G}it{H}ub repository.
\newblock \url{https://github.com/chaim-e/diagram-sampler}.

\bibitem[EZ17]{even2017models}
Chaim Even-Zohar.
\newblock Models of random knots.
\newblock {\em Journal of Applied and Computational Topology}, 1(2):263--296,
  2017.

\bibitem[GL11]{garoufalidis2011asymptotics}
Stavros Garoufalidis and Thang T~Q L{\^e}.
\newblock Asymptotics of the colored {J}ones function of a knot.
\newblock {\em Geometry \& Topology}, 15(4):2135--2180, 2011.

\bibitem[GN92]{grosberg1992algebraic}
A~Grosberg and S~Nechaev.
\newblock Algebraic invariants of knots and disordered {P}otts model.
\newblock {\em Journal of Physics A: Mathematical and General}, 25(17):4659,
  1992.

\bibitem[HHSY12]{hayashi2012minimal}
Chuichiro Hayashi, Miwa Hayashi, Minori Sawada, and Sayaka Yamada.
\newblock Minimal unknotting sequences of {R}eidemeister moves containing
  unmatched {RII} moves.
\newblock {\em Journal of Knot Theory and Its Ramifications}, 21(10):1250099,
  2012.

\bibitem[JP98]{jones1998lissajous}
Vaughan~FR Jones and J{\'o}zef~H Przytycki.
\newblock Lissajous knots and billiard knots.
\newblock {\em Banach Center Publications}, 42:145--163, 1998.

\bibitem[Kau98]{kauffman1997fourier}
Louis~H Kauffman.
\newblock Fourier knots.
\newblock {\em Ideal knots, World Scientific, available at arXiv preprint
  q-alg/9711013}, 19, 1998.

\bibitem[KP11]{koseleff2011chebyshev}
Pierre-Vincent Koseleff and Daniel Pecker.
\newblock Chebyshev knots.
\newblock {\em Journal of Knot theory and its ramifications}, 20(04):575--593,
  2011.

\bibitem[KRJ16]{kappraff2016meanders}
Jay Kappraff, Ljiljana Radovi{\'c}, and Slavik Jablan.
\newblock Meanders, knots, labyrinths and mazes.
\newblock {\em Journal of Knot Theory and Its Ramifications}, 25(09):1641009,
  2016.

\bibitem[Kui87]{kuiper1987new}
Nicolaas~H Kuiper.
\newblock A new knot invariant.
\newblock {\em Mathematische Annalen}, 278(1-4):193--209, 1987.

\bibitem[Lam97]{lamm1997there}
Christoph Lamm.
\newblock There are infinitely many {L}issajous knots.
\newblock {\em Manuscripta Mathematica}, 93(1):29--37, 1997.

\bibitem[Lam12]{lamm2012fourier}
Christoph Lamm.
\newblock Fourier knots.
\newblock {\em preprint arXiv:1210.4543}, 2012.

\bibitem[LC03]{la2003approaches}
Michael La~Croix.
\newblock Approaches to the enumerative theory of meanders,
  \url{http://www.math.uwaterloo.ca/\~malacroi/Latex/Meanders.pdf}.
\newblock {\em Master{'}s Essay, University of Waterloo}, 2003.

\bibitem[Lis57]{lissajous1857memoire}
Jules~Antoine Lissajous.
\newblock {\em M{\'e}moire sur l'{\'e}tude optique des mouvements vibratoires}.
\newblock France, Paris, 1857.

\bibitem[LT79]{lipton1979separator}
Richard~J Lipton and Robert~Endre Tarjan.
\newblock A separator theorem for planar graphs.
\newblock {\em SIAM Journal on Applied Mathematics}, 36(2):177--189, 1979.

\bibitem[Mak89]{makanin1989analogue}
Gennadiy~Semenovich Makanin.
\newblock On an analogue of the {A}lexander--{M}arkov theorem.
\newblock {\em Izvestiya Rossiiskoi Akademii Nauk. Seriya Matematicheskaya},
  53(1):200--210, 1989.

\bibitem[Nec96]{nechaev1996statistics}
Sergei~K Nechaev.
\newblock {\em Statistics of knots and entangled random walks
  \url{https://arxiv.org/abs/cond-mat/9812205}}.
\newblock World Scientific, 1996.

\bibitem[Obe16]{obeidin2016volumes}
Malik Obeidin.
\newblock Volumes of random alternating link diagrams.
\newblock {\em preprint arXiv: 1611.04944}, 2016.

\bibitem[Owa18]{owad2018straight}
Nicholas Owad.
\newblock Straight knots.
\newblock {\em preprint arXiv:1801.10428}, 2018.

\bibitem[Oza07]{ozawa2007edge}
Makoto Ozawa.
\newblock Edge number of knots and links.
\newblock {\em preprint arXiv:0705.4348}, 2007.

\bibitem[Oza10]{ozawa2010waist}
Makoto Ozawa.
\newblock Waist and trunk of knots.
\newblock {\em Geometriae Dedicata}, 149(1):85--94, 2010.

\bibitem[Pri86]{pride1986tits}
Stephen~J Pride.
\newblock On {T}its' conjecture and other questions concerning {A}rtin and
  generalized {A}rtin groups.
\newblock {\em Inventiones Mathematicae}, 86(2):347--356, 1986.

\bibitem[Riv16]{rivin2016random}
Igor Rivin.
\newblock Random space and plane curves.
\newblock {\em preprint arXiv:1607.05239}, 2016.

\bibitem[RJ15]{radovic2015meander}
Ljiljana Radovic and Slavik Jablan.
\newblock Meander knots and links.
\newblock {\em Filomat}, 29(10):2381--2392, 2015.

\bibitem[SV16]{soret2016lissajous}
Marc Soret and Marina Ville.
\newblock Lissajous and {F}ourier knots.
\newblock {\em Journal of Knot Theory and Its Ramifications}, 25(05):1650026,
  2016.

\bibitem[Tan89]{taniyama1989partial}
Kouki Taniyama.
\newblock A partial order of knots.
\newblock {\em Tokyo Journal of Mathematics}, 12(1):205--229, 1989.

\bibitem[Thu]{38813}
Bill Thurston.
\newblock Tying knots with reflecting lightrays, answer.
\newblock MathOverflow question 38813, 2010-09-16.
\newblock \url{https://mathoverflow.net/q/38813}.

\bibitem[Tra95]{trautwein1995harmonic}
Aaron~Keith Trautwein.
\newblock Harmonic knots.
\newblock Ph.D. Thesis, University of Iowa, 1995.

\bibitem[VH60]{von1960arkadenfadendarstellung}
G{\"u}nter Von~Hotz.
\newblock Arkadenfadendarstellung von knoten und eine neue darstellung der
  knotengruppe.
\newblock In {\em Abhandlungen aus dem Mathematischen Seminar der
  Universit{\"a}t Hamburg}, volume~24, pages 132--148. Springer, 1960.

\end{thebibliography}

\end{document}